\documentclass[11pt,reqno]{amsart}

\usepackage{caption,amsmath,amssymb,amsfonts,amsthm,latexsym,graphicx,multirow,color,hyperref,enumerate}
\usepackage[all]{xy}

\newcommand{\A}{\mathrm{A}} \newcommand{\AGL}{\mathrm{AGL}}   \newcommand{\Aut}{\mathrm{Aut}}
 \newcommand{\bbF}{\mathbb{F}}

\newcommand{\Cen}{\mathbf{C}} \newcommand{\calB}{\mathcal{B}}   \newcommand{\calO}{\mathcal{O}}   \newcommand{\Co}{\mathrm{Co}} 
\newcommand{\D}{\mathrm{D}} 

\newcommand{\F}{\mathrm{F}} 
\newcommand{\G}{\mathrm{G}} \newcommand{\GaL}{\mathrm{\Gamma L}}    \newcommand{\GL}{\mathrm{GL}} \newcommand{\GO}{\mathrm{O}} \newcommand{\GU}{\mathrm{GU}}
 \newcommand{\Hol}{\mathrm{Hol}} \newcommand{\HS}{\mathrm{HS}}
\newcommand{\J}{\mathrm{J}}

\newcommand{\lefthat}{\scalebox{1.3}[1]{\text{$\hat{~}$}}}
\newcommand{\M}{\mathrm{M}} \newcommand{\magma}{\textsc{Magma}}
\newcommand{\N}{\mathrm{N}}
\newcommand{\Nor}{\mathbf{N}}
\newcommand{\Out}{\mathrm{Out}}
\newcommand{\Pa}{\mathrm{P}}  \newcommand{\PGL}{\mathrm{PGL}} \newcommand{\PGaL}{\mathrm{P\Gamma L}}   \newcommand{\ppd}{\mathrm{ppd}} \newcommand{\POm}{\mathrm{P\Omega}} \newcommand{\PSL}{\mathrm{PSL}} \newcommand{\PSiL}{\mathrm{P\Sigma L}}   \newcommand{\PSp}{\mathrm{PSp}} \newcommand{\PSU}{\mathrm{PSU}}

 \newcommand{\Rad}{\mathbf{R}}  
 \newcommand{\SL}{\mathrm{SL}}  \newcommand{\Soc}{\mathrm{Soc}} \newcommand{\Sp}{\mathrm{Sp}}  \newcommand{\SU}{\mathrm{SU}} \newcommand{\Suz}{\mathrm{Suz}} \newcommand{\Sy}{\mathrm{S}}  \newcommand{\Sz}{\mathrm{Sz}}

\newcommand{\Z}{\mathbf{Z}} \newcommand{\ZZ}{\mathrm{C}}

\newtheorem{theorem}{Theorem}[section]
\newtheorem{corollary}[theorem]{Corollary}
\newtheorem{lemma}[theorem]{Lemma}
\newtheorem{proposition}[theorem]{Proposition}
\newtheorem{problem}[theorem]{Problem}

\theoremstyle{definition}

\newtheorem{example}[theorem]{Example}
\newtheorem{hypothesis}[theorem]{Hypothesis}
\newtheorem*{remark}{Remark}

\numberwithin{equation}{section}
\numberwithin{table}{section}
\def\a{\alpha}\def\ov{\overline}
\def\calO{{\mathcal O}}

\def\sfM{{\sf M}}

\begin{document}

\title[Exact factorizations]{The Exact Factorizations of Almost Simple Groups}

\author{Cai Heng Li, Lei Wang, Binzhou Xia}
\address[Li]
{SUSTech International Center for Mathematics, and Department of Mathematics, Southern University of Science and Technology\\Shenzhen 518055, Guangdong\\P. R. China
\newline Email: {\tt lich@sustech.edu.cn}
}
\address[Wang]
{School of Mathematics and Statistics\\Yunnan University\\Kunming 650091, Yunnan\\P. R. China
\newline Email: {\tt wanglei@ynu.edu.cn}
}
\address[Xia]
{School of Mathematics and Statistics\\The University of Melbourne\\Parkville, VIC 3010\\Australia
\newline Email: {\tt binzhoux@unimelb.edu.au}
}

% \date{\today}

\begin{abstract}
This paper presents a classification of exact factorizations of almost simple groups, which has been a long-standing open problem initiated around 1980 by the work of Wiegold-Williamson, and significantly progressed by Liebeck, Praeger and Saxl in 2010.
%In fact, we slightly extend the classification for classical groups to the so-called near-exact factorizations.
The classification is then used to solve problems in bicrossproduct Hopf algebras and permutation groups.
%
%Characterizing primitive permutation groups $G$ containing a regular subgroup $H$ is a classical problem in permutation group theory.
%The problem has been solved for various important special cases, and among them a significant result achieved by Liebeck, Praeger and Saxl in 2010 solves the problem for the case where $G$ is primitive and almost simple.
%
%Praeger developed a theory for quasiprimitive permutation groups in 1992, which extended research scope of permutation group theory.
%Many problems in the applications of permutation groups can be reduced to the quasiprimitive case instead of the primitive case.
%, and thus quasiprimitive permutation groups have been extensively studied during the past 30 years.
%In this paper, we classify the pairs $H<G$ for two cases, namely, $G$ is almost simple, or $G$ is quasiprimitive and $H$ is simple, which extend the work by Liebeck, Praeger and Saxl in 2010 on the primitive case.
%One would expect that these results will have many applications, for instance, to Cayley graphs on simple groups in subsequent work.

%\vspace{5\baselineskip}

\vskip0.1in


\end{abstract}

\maketitle

\section{Introduction}

An expression $G=HK$ of a group $G$ as the product of subgroups $H$ and $K$ is called a \emph{factorization} of $G$, where $H$ and $K$ are called \emph{factors}. A group factorization is said to be \emph{exact} if the intersection of the two factors is trivial.
Recall that a finite group $G$ is said to be \emph{almost simple} if $S\leqslant G\leqslant\Aut(S)$ for some nonabelian simple group $S$, where $S=\Soc(G)$ is the \emph{socle} of $G$.
The following is a long-standing problem in group theory.

\begin{problem}\label{PrbXia1}
Classify exact factorizations of finite almost simple groups.
\end{problem}

Indeed, Problem~\ref{PrbXia1} was initiated around 1980, when Wiegold and Williamson~\cite{WW1980} determined the exact factorizations of alternating groups and symmetric groups. A significant result for the problem was achieved in 2010 by Liebeck, Praeger and Saxl~\cite{LPS2010}, who determined exact factorizations of finite almost simple groups with one of the factors maximal.
Recently, Problem~\ref{PrbXia1} was solved in~\cite{LX2022} and~\cite{BL2021} in the case where one of the factors is solvable.
The problem is thus reduced to the case where both factors are nonsolvable, and a main result of this paper is to present a complete solution as in Theorem~\ref{ThmExact}. Our notation follows~\cite{LX2022}.

\begin{theorem}\label{ThmExact}
Let $G$ be a finite almost simple group, and let $H$ and $K$ be nonsolvable subgroups of $G$. Then $G=HK$ is an exact factorization if and only if $(G,H,K)$ lies in Table~$\ref{TabXia2}$ (with $H$ and $K$ possibly interchanged).
\end{theorem}

\begin{table}[htbp]
\captionsetup{justification=centering}
\caption{Exact factorizations of almost simple groups with nonsolvable factors}\label{TabXia2}
\begin{tabular}{|l|l|l|l|l|}
\hline
Row~& $G$ & $H$ & $K$ & Remarks\\
\hline
1 & $\A_n.\calO$ & $[n].\calO_1$ & $\A_{n-1}.\calO_2$ & $H\cap\A_n=[n]$ if $\calO_1=2$\\
2 & $\A_{q+1}.\calO$ & $\PSL_2(q).(d\times\calO_1)$ & $\A_{q-2}.\calO_2$ & Example~\ref{ExaRow2}\\
3 & $\A_{q+1}.\calO$ & $\PSL_2(q)$ & $\A_{q-2}.(2\times\calO)$ & $q\equiv3\pmod{4}$\\
\hline
4 & $\A_9.\calO$ & $\PSL_2(8)$ & $(\A_5\times3).(2\times\calO)$ & $k=1$\\
5 & $\A_9.\calO$ & $\PGaL_2(8)$ & $\Sy_5\times\calO$ & $k=|\calO|$\\
6 & $\A_{11}.\calO$ & $\M_{11}$ & $\A_7.\calO$ & $k=3$ if $\calO=2$\\
7 & $\A_{12}.\calO$ & $\M_{12}$ & $\A_7.\calO$ & $k=3$ if $\calO=2$\\
8 & $\A_{33}.\calO$ & $\PGaL_2(32)$ & $(\A_{29}\times3).(2\times\calO)$ & \\
9 & $\A_{121}.\calO$ & $11^2{:}\SL_2(5)$ & $\A_{119}.\calO$ & $H$ sharply $2$-transitive\\
10 & $\A_{841}.\calO$ & $29^2{:}(\SL_2(5)\times7)$ & $\A_{839}.\calO$ & $H$ sharply $2$-transitive\\
11 & $\A_{3481}.\calO$ & $59^2{:}(\SL_2(5)\times29)$ & $\A_{3479}.\calO$ & $H$ sharply $2$-transitive\\
\hline
12 & $\SL_4(4).(2\times\calO)$ & $(\SL_2(16).4)\times\calO_1$ & $\SL_3(4).(2\times\calO_2)$ & Example~\ref{ExaRow12}\\
13 & $\SU_4(4).4$ & $\SL_2(16).4$ & $\SU_3(4).4$ & Example~\ref{ExaRow13}\\
14 & $\Sp_4(4).2$ & $\Sy_5$ & $\SL_2(16).4$ & $K$ maximal, $h=1$\\
15 & $\Sp_6(2)$ & $\Sy_5\times\calO$ & $\SU_3(3).(2/\calO)$ & Example~\ref{ExaRow15}\\
16 & $\Sp_6(2)$ & $2^4{:}\A_5$ & $\SL_2(8).3$ & $K$ maximal\\
17 & $\Sp_6(4).2$ & $\SL_2(16).4$ & $\G_2(4).2$ & $K$ maximal, $h=2$\\
18 & $\Sp_8(2)$ & $\Sy_5$ & $\GO_8^-(2)$ & $K$ maximal, $h=1$\\
19 & $\Omega_8^+(2).\calO$ & $2^4{:}\A_5.\calO_1$ & $\A_9.\calO_2$ & Example~\ref{ExaRow19}(a)\\
20 & $\Omega_8^+(2).\calO$ & $\Sy_5\times\calO_1$ & $\Sp_6(2)\times\calO_2$ & Example~\ref{ExaRow19}(b)\\
21 & $\Omega_8^+(4).(2\times\calO)$ & $(\SL_2(16).4)\times\calO_1$ & $(\Sp_6(4).2)\times\calO_2$ & Example~\ref{ExaRow21}\\
\hline
22 & $\M_{24}$ & $2^4{:}\A_7$ & $\PSL_2(23)$ & $K$ maximal\\
23 & $\M_{24}$ & $\PSiL_3(4)$ & $\PSL_2(23)$ & $K$ maximal\\
\hline
\end{tabular}

~\\
Symbols in this table:
{\small %$\calO\leqslant\mathrm{C}_2$,
$[n]$ denotes an unspecified group of order $n$;
$\calO_1,\calO_2\leqslant\calO\leqslant\ZZ_2$ with $|\calO|=|\calO_1||\calO_2|$;
$q\geqslant7$ is a prime power, and $d=(2,q-1)$.\
The parameter $k$ (or $h$) is the number of conjugacy classes of $K$ (or $H$) for each $H$ (or each $K$, respectively).\
%The parameter $h$ is the number of conjugacy classes of $H$ for each $K$.\
In row~1, $n\geqslant60$, $[n]$ is a regular permutation group on $n$ points;
in row~3, $H$ is $3$-homogeneous, and $k=1$ if $q=7$; in rows~9--11, $k=2$ if $\calO=2$.}
%The term $t$-tr is short for $t$-transitive.
\end{table}

A factorization of an almost simple group is called \emph{nontrivial} if both of the factors are core-free.
Obviously, an exact factorization of an almost simple group with nonsolvable factors is indeed nontrivial.
Determining all (nontrivial) factorizations of almost simple groups is a fundamental problem in the theory of simple groups.
The factorizations of almost simple exceptional groups of Lie type were classified by Hering, Liebeck and Saxl~\cite{HLS1987}
% \footnote{In part~(b) of Theorem~2 in~\cite{HLS1987}, $A_0$ can also be $\G_2(2)$, $\SU_3(3)\times2$, $\SL_3(4).2$ or $\SL_3(4).2^2$ besides $\G_2(2)\times2$.}
in 1987.
For the other families of almost simple groups, a landmark result was achieved by Liebeck, Praeger and Saxl~\cite{LPS1990} in 1990, which classifies the maximal (not necessarily exact) factorizations of almost simple groups. (A factorization is said to be \emph{maximal} if both the factors are maximal subgroups.)
Then factorizations of alternating and symmetric groups are classified in~\cite{LPS1990}, and factorizations of sporadic almost simple groups are classified in~\cite{Giudici2006}. Therefore, the essential case to deal with for Problem~\ref{PrbXia1} is that of classical groups.

Let $L$ be a finite simple group, and let $\sfM(L)$ be the Schur multiplier of $L$.
Then $\sfM(L)$ is explicitly determined (see \cite[Table~4.1]{Gorenstein1982}) and has few prime divisors. Let
\[\pi=\pi(\sfM(L)),\ \mbox{the set of prime divisors of the order $|\sfM(L)|$}.\]
A factorization $G=HK$ with $\Soc(G)=L$ is called {\it near-exact} if $H\cap K$ is a $\pi$-group.
Thus an exact factorization is a near-exact factorization.
%
%, and let $H$ and $K$ be subgroups of $H$ not containing $L$. If $G=HK$ such that $H\cap K$ is a $\pi$-group, then we call the factorization $G=HK$ a {\it near-exact} factorization.
We shall prove the following theorem.

\begin{theorem}\label{ThmNearExact}
Let $G$ be an almost simple classical group of Lie type, and let $G=HK$ be a nontrivial near-exact factorization with nonsolvable $H$ and $K$.
Then $\Soc(G)$ is one of:
\[
\PSL_3(4),\SL_4(2),\SL_4(4),\PSU_4(3),\SU_4(4),\Omega_7(3),\Omega_8^+(2),\Omega_8^+(4),
\Sp_4(4),\Sp_6(2),\Sp_6(4),\Sp_8(2).
\]
%Moreover, the triples $(G,H,K)$ are presented in Propositions~\emph{\ref{ThmLinear},~\ref{ThmUnitary},~\ref{ThmOmega},~\ref{ThmOmegaPlus},~\ref{ThmSymplectic}}.
\end{theorem}

\begin{remark}
Information on the factorizations $G=HK$ in Theorem~\ref{ThmNearExact} can be found in Propositions~\ref{ThmLinear}, \ref{ThmUnitary}, \ref{ThmOmega}, \ref{ThmOmegaPlus}, \ref{ThmSymplectic}.
\end{remark}

Although near-exact factorizations are a genuine extension of exact factorizations, the proof of Theorem~\ref{ThmNearExact} is much shorter than the proof of Theorem~\ref{ThmExact} in an earlier unpublished version of this paper~\cite{LWX}, which directly addresses exact factorizations.
This is the reason why we extend our work to classify near-exact factorizations.
%The somewhat surprising fact is that considering near-exact factorizations of classical groups enables us to use mathematical induction on this family of factorizations, and thus significantly shorten the proof.

% {\bf Remarks:}
%The general factorization problem for almost simple classical groups is still a challenge problem, although it has been solved for some special cases in the preprints \cite{LWX}.
%Relatively, the argument presented here for near-exact factorizations is a bit surprisingly short.

In 2000, Etingof, Gelaki, Guralnick and Saxl~\cite{EGGS2000} used the bicrossproduct Hopf algebra associated with the exact factorization $\M_{24}=(2^4{:}\A_7)\PSL_2(23)$ to give the first example of biperfect semisimple Hopf algebra of dimension greater than $1$, which yields a negative answer to~\cite[Question 7.5]{EG2000}. In fact, both factors of the factorization $\M_{24}=(2^4{:}\A_7)\PSL_2(23)$ are perfect and self-normalizing, and it was shown in~\cite[Theorem~2.3]{EGGS2000} that the bicrossproduct Hopf algebra associated with an exact factorization $G=HK$ is biperfect if and only if both $H$ and $K$ are perfect and self-normalizing. The authors in~\cite{EGGS2000} described the existence of such an exact factorization as ``amazing'', and suspected that it would be the only example among finite simple groups.
% This is confirmed by the following corollary
The following corollary of Theorem~\ref{ThmExact} confirms that this conjectured result is true.

\begin{corollary}\label{CorBiperfect}
The only exact factorization of finite simple group with both factors perfect and self-normalizing is $\M_{24}=(2^4{:}\A_7)\PSL_2(23)$.
\end{corollary}

Permutation groups containing a regular subgroup form an important class of permutation groups which has been studied for a long time, dated back to Burnside in the 19th century. For a review of the literature and background of this topic, the reader is referred to the introduction of~\cite{LPS2010}.

A permutation group acting on a set is said to be \emph{primitive} if it does not preserve any nontrivial partition of the set, and is said to be \emph{quasiprimitive} if each of its nontrivial normal subgroups is transitive. Primitive groups are all quasiprimitive, but the converse statement is not true.
Liebeck, Praeger and Saxl in their memoirs~\cite{LPS2010} studied finite primitive permutation groups $G$ with a regular subgroup, and completely classified them under the assumption that $G$ is almost simple.
%In this paper, we extend their classification to a general almost simple group $G$, without assuming that $G$ is primitive.
%
%The problem of determining regular subgroups of a permutation group can be stated in the language of group factorizations.
%
%Another result in the above-mentioned memoirs of Liebeck, Praeger and Saxl is the classification of finite primitive groups with a regular almost simple subgroup~\cite[Theorem~1.4]{LPS2010}.
%This has a very nice consequence regarding Cayley graphs of finite simple groups.
As another application of Theorem~\ref{ThmExact}, we classify finite quasiprimitive permutation groups containing a regular simple subgroup.

\begin{theorem}\label{qp>simple-gps}
Let $X$ be a finite quasiprimitive group on $\Omega$, and let $L$ be a nonabelian simple regular subgroup of $X$.
Then one of the following holds:
\begin{enumerate}[{\rm (a)}]
\item $L\lhd X\leqslant\Aut(L)$;
\item $L\times L\lhd X\leqslant (L\times L).(\Out(L)\times\Sy_2)$;
\item $X$ is an almost simple group, and $X=LX_\a$ is an exact factorization, given in Table~$\ref{TabXia2}$ and~\cite[Theorem~3]{BL2021}.
\end{enumerate}
\end{theorem}

In particular, Theorem~\ref{qp>simple-gps} shows that a finite quasiprimitive permutation group of product action type does not contain a regular simple group.
Theorem~\ref{qp>simple-gps} contributes to the inclusion problem of quasiprimitive permutation groups, which was studied by Praeger and has been significantly progressed \cite{Baddeley-Praeger-2003,Praeger1990,Praeger2003}.

%In particular, Theorem~\ref{qp>simple-gps} shows that a finite quasiprimitive permutation group of product action type does not contain a regular simple group, and this theorem will be applied to study Cayley graphs of simple groups in subsequent work.

The layout of the paper is as follows. After this Introduction, we collect some preliminary results in Section~\ref{SecXia3}.
In Section~\ref{SecXia7} we prove Theorem~\ref{ThmNearExact} (the proof consists of Propositions~\ref{ThmLinear},~\ref{ThmUnitary},~\ref{ThmOmega},~\ref{ThmOmegaPlus} and \ref{ThmSymplectic}). Based on this, we prove Theorem~\ref{ThmExact} in Section~\ref{SecXia6}.
Then in Section~\ref{SecHopf} we apply Theorem~\ref{ThmExact} to Hopf algebras, proving Corollary~\ref{CorBiperfect}.
In the final section, we discuss transitive simple groups of quasiprimitive permutation groups and prove Theorem~\ref{qp>simple-gps}.

\section{Preliminaries}\label{SecXia3}

%This section presents notation and preliminary results that will be used later.
The basic group-theoretic notation in this paper is standard, see for example \cite{CCNPW1985} and \cite{LPS1990}.
A core-free subgroup $H$ of an almost simple group $G$ is said to be a \emph{factor} of $G$ if there exists a core-free supplement $K$ of $H$ in $G$.

Let $a$ and $b$ be positive integers. Denote the greatest common divisor of $a$ and $b$ by $(a,b)$. For a positive integer $n$ and a prime $p$, denote by $n_p$ the largest $p$-power dividing $n$, and denote $n_{p'}=n/n_p$.
If $\pi$ is a set of primes, then we denote $n_\pi=\prod_{p\in\pi}n_p$ and $n_{\pi'}=n/n_\pi$.
%For convenience, we say that $n$ is divisible by $a/b$ if and only if $nb$ is divisible by $a$. For example, $5$ is divisible by $5/2$ but not by $3/2$.

Let $a>1$ and $k>1$ be integers. A prime $r$ is called a \emph{primitive prime divisor} of the pair $(a,k)$ if $r$ divides $a^k-1$ but does not divide $a^j-1$ for any positive integer $j<k$. %We record the celebrated Zsigmondy's theorem here on the existence of primitive prime divisors.

\begin{theorem}[Zsigmondy~\cite{Zsigmondy1892}]
If $k\geqslant2$, then $(a,k)$ has a primitive prime divisor except for either $k=2$ and $a=2^e-1$ is a Mersenne prime, or $(a,k)=(2,6)$.
%Moreover, if $r$ is a primitive prime divisor of $(a,k)$ then $r>k$.
\end{theorem}

If $a$ is prime and $(a,k)\neq(2,6)$, then a primitive prime divisor of the pair $(a,k)$ is simply called a \emph{primitive prime divisor} of $a^k-1$, and we let
\[
\ppd(a^k-1)\text{ denote the set of primitive prime divisors of }(a,k).
\]
Note that our $\ppd$ function is defined on integers that equal a prime power minus $1$.
For example, $\ppd(3^6-1)=\ppd(9^3-1)=\{7\}$, which is different from the set $\{7,13\}$ of primitive prime divisors of the pair $(9,3)$. For convenience we let
\[
\ppd(63)=\{7\}.
\]
An integer is said to be divisible by or coprime to $\ppd(p^b-1)$ if it is divisible by or coprime to all the primes in $\ppd(p^b-1)$, respectively. %and that an integer is coprime to $\ppd(p^b-1)$ if it is coprime to all the primes in $\ppd(p^b-1)$.

%For a prime number $p$ and an integer $b\geqslant3$, we call a primitive prime divisor of the pair $(p,b)$ simply a \emph{primitive prime divisor} of the number $p^b-1$. By Zsigmondy's theorem, $p^b-1$ has a primitive prime divisor whenever $p^b-1\neq63$. Let $\ppd(p^b-1)$ denote the set of primitive prime divisors of $p^b-1$ if $p^b-1\neq63$, and let $\ppd(63)=\{7\}$. We say that an integer is divisible by $\ppd(p^b-1)$ if it is divisible by all the primes in $\ppd(p^b-1)$, and that an integer is coprime to $\ppd(p^b-1)$ if it is coprime to all the primes in $\ppd(p^b-1)$.

Checking the orders of outer automorphism groups of finite simple groups leads to the following conclusion (see~\cite[Page~38, Proposition~B]{LPS1990}):

\begin{lemma}\label{LemXia6}
Let $L$ be a simple group of Lie type over $\bbF_q$, and let $k\geqslant3$. If $(q,k)\neq(2,6)$, then none of the primitive prime divisors of $(q,k)$ divides $|\Out(L)|$.
\end{lemma}

For a group $X$, let $\Rad(X)$ denote the \emph{solvable radical} of $X$, which is the largest solvable normal subgroup of $X$.
For an almost simple group $G$ with socle $L$ and $X<G$, the radical $\Rad(X\cap L)=\Rad(X)\cap L$, and so $|\Rad(X)|/|\Rad(X\cap L)|=|\Rad(X)L/L|$ divides $|\Out(L)|$.

If $G=HK$ is a group factorization and $A$ and $B$ are subgroups of $G$ containing $H$ and $K$ respectively, then
\[
A=H(A\cap B)\quad\text{and}\quad B=(A\cap B)K.
\]
This observation will be useful especially when $G$ is an almost simple group with maximal subgroups $A$ and $B$.
For a maximal factorization $G=AB$ (given in \cite[Theorem~A]{LPS1990}), by Lemma~\ref{LemXia6}, inspecting $|\Rad(A\cap L)|$ and $|\Rad(B\cap L)|$ leads to the next lemma.
%for the triple $(L,A\cap L,B\cap L)$, where $G=AB$ is a maximal factorization (see~\cite[Theorem~A]{LPS1990}).

\begin{lemma}{\rm(\cite[Proposition~2.17]{LX2022})}\label{LemXia7}
Let $G$ be an almost simple classical group of Lie type over $\bbF_q$, and let $G=AB$ be a maximal factorization, with both $A$ and $B$ nonsolvable.
\begin{enumerate}[{\rm(a)}]
\item For $X=A$ or $B$, if $X$ has exactly one nonsolvable composition factor, then $X/\Rad(X)$ is almost simple, and $(A\cap B)\Rad(X)/\Rad(X)$ is core-free in $X/\Rad(X)$.
%If $A$ has exactly one nonsolvable composition factor, then $A/\Rad(A)$ is almost simple, and $(A\cap B)\Rad(A)/\Rad(A)$ is core-free in $A/\Rad(A)$.
%Similarly, if $B$ has exactly one nonsolvable composition factor, then $B/\Rad(B)$ is almost simple, and $(A\cap B)\Rad(B)/\Rad(B)$ is core-free in $B/\Rad(B)$.

\item If $k\geqslant3$ and $(q,k)\neq(2,6)$, then no primitive prime divisor of $(q,k)$ divides $|\Rad(A)||\Rad(B)|$.

\end{enumerate}
\end{lemma}

%If $G$ is an almost simple group with socle $L$ and $X$ is a subgroup of $G$, then $\Rad(X\cap L)=\Rad(X)\cap L$, and so $|\Rad(X)|/|\Rad(X\cap L)|=|\Rad(X)L/L|$ divides $|\Out(L)|$.
%Then, by Lemma~\ref{LemXia6}, we obtain the following lemma by checking $|\Rad(A\cap L)|$ and $|\Rad(B\cap L)|$ for the triple $(L,A\cap L,B\cap L)$, where $G=AB$ is a maximal factorization (see~\cite[Theorem~A]{LPS1990}).

%\begin{lemma}\label{LemXia22}
%Let $G$ be an almost simple group with socle $L$, where $L$ is a simple group of Lie type over $\bbF_q$, and let $G=AB$ be a maximal factorization with nonsolvable  $A$ and $B$.
%If $k\geqslant3$ and $(q,k)\neq(2,6)$, then no primitive prime divisor of $(q,k)$ divides $|\Rad(A)|$ or $|\Rad(B)|$.
%\end{lemma}

The above two lemmas will be used repeatedly and implicitly in this paper.

\begin{lemma}\label{LemXia5}
Let $G=HK$ be an almost simple group with socle $L$. Then there exist a group $G^*$ with $L\lhd G^*\leqslant G$, $H^*\leqslant H$ and $K^*\leqslant K$ such that $G^*=H^*K^*$, $H^*\cap L=H\cap L$, $K^*\cap L=K\cap L$ and
\[
H^*L=K^*L=G^*
\]
In particular, if $H\cap K$ is a $\pi$-group, then so is $H^*\cap K^*$.
\end{lemma}

\begin{proof}
Take $G^*=HL\cap KL$, $H^*=H\cap G^*$ and $K^*=K\cap G^*$. Then $L\trianglelefteq G^*\leqslant G$, $H^*\leqslant H$, $K^*\leqslant K$, $H^*\cap L=H\cap L$ and $K^*\cap L=K\cap L$. Since $H$ and $K$ are both core-free in $G$, it follows that $H^*$ and $K^*$ are both core-free in $G^*$. Moreover, it is easy to verify that $G^*=H^*K^*$ and $H^*L=K^*L=G^*$ (see for instance~\cite[Lemma~2(i)]{LPS1996}).
\end{proof}

To exclude some candidates for factorizations $G=HK$, Lemma~\ref{LemXia5} enables us to assume
\begin{equation}\label{EqnXia6}
HL=KL=G.
\end{equation}
%Moreover, since Lemma~\ref{LemXia5} asserts that $H^*\leqslant H$ and $K^*\leqslant K$, we have that, if $H\cap K$ is a $\pi$-group, then so is $H^*\cap K^*$.
Under this assumption, the next lemma enables us to embed a near-exact factorization into a maximal factorization.

\begin{lemma}\label{LemXia27}
Let $G$ be an almost simple group with socle $L$, and let $H$ and $K$ be core-free subgroups of $G$ such that $HL=KL=G$. Then there exist core-free maximal subgroups $A$ and $B$ of $G$ such that $H\leqslant A$ and $K\leqslant B$.
\end{lemma}

\begin{proof}
Let $A$ be a maximal subgroup of $G$ containing $H$.
If $A\geqslant L$, then $A\geqslant HL=G$ since $H\leqslant A$, a contradiction.
Hence $A$ is core-free in $G$.
Similarly, any maximal subgroup $B$ of $G$ containing $K$ is core-free in $G$.
\end{proof}

As usual, for a finite group $X$, we denote by $X^{(\infty)}$ the smallest normal subgroup of $X$ such that $X/X^{(\infty)}$ is solvable.
Note that $H^{(\infty)}=(H\cap L)^{(\infty)}$ and $K^{(\infty)}=(K\cap L)^{(\infty)}$.

\begin{lemma}\label{LemXia4}
Let $G$ be an almost simple group with socle $L$, and let $G=HK$ be a factorization with $HL=KL=G$. Then $|L||H\cap K|=|G/L||H\cap L||K\cap L|$. In particular, if $H\cap K$ is a $\pi$-group for some set $\pi$ of primes, then $|H^{(\infty)}|_{\pi'}|K^{(\infty)}|_{\pi'}$ divides $|L|_{\pi'}$.
\end{lemma}

\begin{proof}
We deduce from $G=HK=HL=KL$ that
\begin{align*}
|G||H\cap K|=|H||K|,\\
|H||L|=|G||H\cap L|,\\
|K||L|=|G||K\cap L|.
\end{align*}
Multiplying the three equations together on both sides, we obtain
\[
|L|^2|H\cap K|=|G||H\cap L||K\cap L|
\]
and thus $|L||H\cap K|=|G/L||H\cap L||K\cap L|$.
\end{proof}

Finally, we introduce some notation regarding classical groups acting on subspaces.

Let $T$ be a classical linear group on $V$ with center $Z$ such that $T/Z$ is a classical simple group, and let $S$ be a subgroup of $\GaL(V)$ containing $T$ as a normal subgroup. Then for a subgroup $X$ of $S$, denote by $\lefthat X$ the subgroup $(X\cap T)Z/Z$ of $T/Z$.

%Our notation follows~\cite{LX} if there is no further instruction.
Let $G$ be a classical group defined on a vector space $V$. If $G$ is a linear group, define $\Pa_i[G]$ to be the stabilizer in $G$ of an $i$-space in $V$. Now suppose that $G$ is a symplectic, unitary or orthogonal group. If $G$ is not a plus type orthogonal group of dimension $2i$ or $2i+2$, then define $\Pa_i[G]$ to be the stabilizer of a totally singular $i$-space in $G$. If $G$ is a plus type orthogonal group of dimension $2i$, then $G^{(\infty)}$ has precisely two orbits on the set of totally singular $i$-spaces and define $\Pa_{i-1}[G]$ and $\Pa_i[G]$ to be the stabilizers in $G$ of totally singular $i$-spaces in the two different orbits. Let $W$ be a nondegenerate $i$-space in $V$, and
\begin{enumerate}[{\rm (i)}]
\item[(i)] $\N_i[G]=G_W$ if $G$ is symplectic, unitary, or orthogonal of even dimension with $i$ odd;
\item[(ii)] $\N_i^\varepsilon[G]=G_W$ for $\varepsilon\in\{+,-\}$ if $G$ is orthogonal and $W$ has type $\varepsilon$;
\item[(iii)] $\N_i^\varepsilon[G]=G_W$ for $\varepsilon\in\{+,-\}$ if $G$ is orthogonal of odd dimension and $W^\perp$ has type $\varepsilon$.
\end{enumerate}
For the above defined groups $\Pa_i[G]$, $\N_i[G]$, $\N_i^+[G]$ and $\N_i^-[G]$, we will simply write $\Pa_i$, $\N_i$, $\N_i^+$ and $\N_i^-$, respectively, if the group $G$ is clear from the context.

For a classical group $G$, let $\mathcal{P}_i[G]$, $\mathcal{N}_i[G]$, $\mathcal{N}_i^+[G]$ and $\mathcal{N}_i^-[G]$ be the set of right cosets of $\Pa_i[G]$, $\N_i[G]$, $\N_i^+[G]$ and $\N_i^-[G]$, respectively, in $G$. Note that these sets of right cosets can be identified with the set of certain subspaces. For example, $\mathcal{P}_1[\PSL_n(q)]$ is the set of $1$-spaces of a vector space of dimension $n$ over $\bbF_q$, and $\mathcal{N}_1[\PSU_n(q)]$ is the set of nonsingular $1$-spaces of a unitary space of dimension $n$ over $\bbF_{q^2}$.

%{\color{red}
%The subgroups of $\GaL_n(q)$ that are transitive on the set of $1$-spaces were classified by Hering (see~\cite[Lemma~3.1]{LPS2010}). An \emph{antiflag} of a vector space $V$ is an unordered pair $\{U,W\}$, where $U$ is a $1$-space in $V$ and $W$ is a hyperplane in $V$ not containing $U$. The antiflag-transitive subgroups of $\GaL_n(q)$ were classified by Cameron and Kantor~\cite{CK1979} (see~\cite{Kantor} for a correction), and in this paper we will adopt a slightly different version of the classification from~\cite[Theorem~3.2]{LPS2010}. For some other classical groups $G$, subgroups of $G$ that are transitive on a $G$-orbit of certain subspaces are described in~\cite[Chapter~4]{LPS2010}, which will used in this paper. Recently, the work in~\cite[Chapter~4]{LPS2010} has been extended by Giudici, Glasby and Praeger~\cite{GGP}.
%}

\section{Near-exact factorizations of classical groups}\label{SecXia7}

In this section, we prove Theorem~\ref{ThmNearExact}.
For convenience, we make a hypothesis.
%
%\begin{theorem}\label{thm:Near-exact-F}
%Let $G$ be an almost simple classical group of Lie type, and let $G=HK$ be a nontrivial near-exact factorization with nonsovable $H$ and $K$.
%Then $\Soc(G)$ is one of:
%\[
%\PSL_3(4),\,\PSL_4(2),\,\PSU_4(3),\,\SU_4(4),\,\Omega_7(3),\,\Omega_8^+(2),\,\Omega_8^+(4),\,\Sp_4(4),\,\Sp_6(2),\,\Sp_6(4),\,\Sp_8(2).
%\]
%Moreover, the triples $(G,H,K)$ are shown in Propositions~\emph{\ref{ThmLinear},~\ref{ThmUnitary},~\ref{ThmOmega},~\ref{ThmOmegaPlus},~\ref{ThmSymplectic}}.
%\end{theorem}
%

%The proof of this theorem will be given by the following subsections according to different socles of $G$.
%As mentioned in the Introduction, we will apply mathematical induction in the proof.
%Thus we make the following hypothesis, which will enable us to finish the proof by contradiction in some cases.

\begin{hypothesis}\label{hypo}
Let $G$ be an almost simple classical group with socle $L$, and let $\pi=\pi(\sfM(L))$, where $\sfM(L)$ is the Schur multiplier of $L$.
Let $G=HK$ be a near-exact factorization with nonsolvable $H$ and $K$, namely, $H\cap K$ is a $\pi$-group.
Suppose that $G$ is a minimal (with respect to order) counterexample to the statement of Theorem~\ref{ThmNearExact}.
\end{hypothesis}

This hypothesis serves as our inductive hypothesis since every near-exact factorization of an almost simple classical group of order smaller than $|G|$ satisfies Theorem~\ref{ThmNearExact}.

% \begin{lemma}\label{recursive}
% Under Hypothesis~$\ref{hypo}$, the following statements hold:if , then $L$ is not a linear group or an even-dimenstional orthogonal group, then one of the following holds (with $A$ and $B$ interchanged if necessary):
% \begin{enumerate}[{\rm (a)}]
% \item
% if $\pi(\Rad(A))\subseteq\pi(\sfM(\Soc(A/\Rad(A))))$, then letting $\overline{\phantom{x}}$ denote the quotient modulo $\Rad(A)$, either $H^{(\infty)}=A^{(\infty)}$ or $\overline{A}=\overline{H}\,\overline{A_\delta}$ is a near-exact factorization satisfying the conclusion of Theorem~$\ref{ThmNearExact}$;
% \item
% if $\pi(\Rad(B))\subseteq\pi(\sfM(\Soc(B/\Rad(B))))$, then letting $\overline{\phantom{x}}$ denote the quotient modulo $\Rad(B)$, either $K^{(\infty)}=B^{(\infty)}$ or $\overline{B}=\overline{K}\,\overline{B_\omega}$ is a near-exact factorization satisfying the conclusion of Theorem~$\ref{ThmNearExact}$.
% \end{enumerate}
% \end{lemma}

% \begin{proof}
% Under Hypothesis~$\ref{hypo}$, $G$ is a minimal counterexample to Theorem~$\ref{ThmNearExact}$.
% Then both $H$ and $K$ are nonsolvable.

% If $H^{(\infty)}\not=A^{(\infty)}$, then $A=HA_\delta$ is a near-exact factorization, and satisfies Theorem~$\ref{ThmNearExact}$.
% Thus either $A_\delta$ is solvable and $H$ is given in \cite[Tables\,1.1-1.2]{LX}, or
% $(A,H,A_\delta)$ lies in Table~\ref{tab:near-exact-F}.

% Similar statements hold for part~(c) for $B=KB_\omega$.
% \end{proof}

\subsection{Linear groups}\label{SecXia1}
\ \vspace{1mm}

Let $L=\Soc(G)=\PSL_n(q)$, where $n\geqslant2$ and $(n,q)\not=(2,2)$ or $(2,3)$.
Then the Schur multiplier $\sfM(L)$ is $\ZZ_{(n,q-1)}$, with the following exceptions:
\[
\begin{array}{c|ccccc}
L & \PSL_2(4) & \PSL_2(9) & \PSL_3(2) & \PSL_3(4) & \PSL_4(2) \\ \hline
\sfM(L) & 2 & 6 & 2 & 3\times 4^2 & 2
\end{array}
\]
Thus $\pi=\pi(\ZZ_{(n,q-1)})$, or $\pi\subseteq\{2,3\}$.
Near-exact factorizations of $G$ are determined below.
%The proof of Theorem~\ref{ThmNearExact} for linear groups is given in the following proposition.

\begin{proposition}\label{ThmLinear}
%Let $G$ be an almost simple group with socle $\PSL_n(q)$, and l
Let $G=HK$ be a near-exact factorization with nonsolvable $H$ and $K$.
%a nontrivial near-exact factorization with nonsolvable $H$ and $K$.
Then, interchanging $H$ and $K$ if necessary, $(G,H,K)$ lies in the following table, where $\calO_1,\calO_2\leqslant\calO\leqslant\ZZ_2$ with $|\calO_1||\calO_2|=|\calO|$.
\[
\begin{array}{|llll|}
\hline
G & H & K & H\cap K \\
\hline
\PSL_3(4).(2\times\calO) & \PGL_2(7)\times\calO_1 & \M_{10}{:}\calO_2 & \Sy_3 \\ \hline
\SL_4(2).\calO & \Sy_5\times\calO_1 & \PSL_2(7).\calO_2 & 1 \\
%\SL_4(2).2 & \Sy_5\times2 & \PGL_2(7) & 2\\
\SL_4(2).\calO & \A_5\times\calO & 2^3{:}\SL_3(2) & 2^2\\
\SL_4(2).\calO & \Sy_5\times\calO & 2^3{:}\SL_3(2) & \D_8\\
\SL_4(2).2 & \Sy_5\times2 & \PGL_2(7) & 2\\
\SL_4(2).2 & \Sy_5 & 2^3{:}\SL_3(2) & 2^2\\ \hline
\SL_4(4).(2\times\calO) & \SL_2(16)\times\calO_1 &\SL_3(4).(2\times\calO_2) & 1 \\
\hline
\end{array}
\]
\end{proposition}

\begin{proof}
If $n$ is a prime, then $G=HK$ is determined in~\cite[Theorem~3.3]{LX2022}, giving rise to
\[
(G,H,K)=(\PSL_3(4).(2\times\calO),\ \PGL_2(7)\times\calO_1,\ \M_{10}{:}\calO_2),
\]
where $\calO_1,\calO_2\leqslant\calO\leqslant\ZZ_2$ with $|\calO_1||\calO_2|=|\calO|$.
In this case, $H\cap K=\Sy_3$.

%We now assume that $n$ is not a prime.
For the small groups $G$ with $L=\PSL_4(2)$, $\PSL_4(3)$ or $\PSL_4(4)$, computation in \magma~\cite{BCP1997} shows that $(G,H,K,H\cap K)$ lies in the table of Proposition~\ref{ThmLinear}.

Now assume that $n\geqslant4$ is not prime, and $(n,q)\not=(4,2)$, $(4,3)$ or $(4,4)$.
We aim to derive a contradiction.
%By Lemmas~\ref{LemXia5} and~\ref{LemXia27}, we may assume that $G=HL=KL$ and there exist maximal subgroups $A$ and $B$ of $G$ containing $H$ and $K$ respectively, so that $G=HK=AB$.
By Lemmas~\ref{LemXia5}-\ref{LemXia27}, we may assume that $H\leqslant A$ and $K\leqslant B$ for some  maximal subgroups $A$ and $B$, so that $G=HK=AB$.
The maximal factorizations $G=AB$ are classified in~\cite{LPS1990}, from which we read off the candidates as in Table~\ref{TabMaxLinear}.
{\begin{table}[htbp]
\captionsetup{justification=centering}
\caption{Maximal factorizations of $G$ with $n\geqslant4$ and $(n,q)\not=(4,2)$}\label{TabMaxLinear}
\[\begin{array}{|l|l|l|}
\hline
 A^{(\infty)} & B^{(\infty)} & \\
\hline
\lefthat\SL_a(q^b) & q^{n-1}{:}\SL_{n-1}(q) & ab=n, a\geqslant2, \mbox{$b$ prime} \\
\lefthat\Sp_n(q) & q^{n-1}{:}\SL_{n-1}(q)  & \mbox{$n\geqslant4$ even} \\

 \lefthat\Sp_n(q) & \SL_{n-1}(q) & \mbox{$n\geqslant4$ even} \\
 \lefthat\SL_{n/2}(q^2) & \SL_{n-1}(q) & \mbox{$n\geqslant4$ even, $q\in\{2,4\}$} \\
\hline
\end{array}\]
\end{table}}

We analyze the subgroup $K$ of $B$.
For a subgroup $X$ of $B$, we denote by $\ov X$ the subgroup $X\Rad(B)/\Rad(B)$ of $B/\Rad(B)$.
%so let $\ov B=B/\Rad(B)$, $\ov K=K\Rad(B)/\Rad(B)$ and $\overline{A\cap B}=(A\cap B)\Rad(B)/\Rad(B)$.
Since $G=AB$, one has $B=(A\cap B)K$, and hence $\ov B=\overline{(A\cap B)}\,\ov K$.
We note that $\ov B$ is an almost simple group with socle $\PSL_{n-1}(q)$, and $\ov K$ is nonsolvable as $K$ is nonsolvable.

Suppose that $\ov K^{(\infty)}<\ov B^{(\infty)}$.
Then $\ov K$ is a factor of $\ov B$ in the nontrivial factorization $\ov B=\overline{(A\cap B)}\,\ov K$.
Since $|\Rad(B)|$ divides $q^{n-1}(q-1)$, we observe that,
\begin{equation}\label{EqnXia12}
\text{if ${|L|/|A^{(\infty)}|}$ is divisible by $\ppd(q^i-1)$ for $i\geqslant2$, then so is $|\ov K^{(\infty)}|$.}
\end{equation}
%
% with $|KN|$ divisible by ${|L|\over|A^{(\infty)}|}$, by Lemmas~\ref{LemXia3} and~\ref{LemXia7}\,(1).
%Thus $|KN|$ is divisible by $r\in\ppd(q^{n-1}-1)$, and
%{\color{blue}
%so is $|\ov K|$ as $|\Rad(B)|$ is coprime to $r$ by Lemma~\ref{LemXia7}\,(2).
%}
Take $r\in\ppd(q^{n-1}-1)$. Then~\eqref{EqnXia12} implies that $|\ov K^{(\infty)}|$ is divisible by $r$.
If $n-1$ is prime, then $\ov B$ (with socle $\PSL_{n-1}(q)$) does not have a nonsolvable core-free factor of order divisible by $r$ by \cite[Theorem\,3.3]{LX2022}, which is a contradiction.
Therefore, $n-1$ is not prime, and as $n$ is not prime either, we have $n\geqslant9$.
Take $s\in\ppd(q^{n-2}-1)$ if $A^{(\infty)}$ is in row~1 of Table~\ref{TabMaxLinear} with $b\geqslant3$, and take $s\in\ppd(q^{n-3}-1)$ otherwise.
Then ${|L|/|A^{(\infty)}|}$ is divisible by $s$, and so is $|\ov K|$ by~\eqref{EqnXia12}.
%{\color{blue}
%Hence $|\ov K|$ is divisible by $s$ (by Lemma~\ref{LemXia7}\,(2)).
%}
However, Table~\ref{TabMaxLinear} shows that there is no core-free factor of $\ov B$ with order divisible by both $r$ and $s$, a contradiction.

We thus conclude that $\ov K^{(\infty)}=\ov B^{(\infty)}=\PSL_{n-1}(q)$. Consequently,
\[
\text{$K^{(\infty)}= q^{n-1}{:}\SL_{n-1}(q)$ or $\SL_{n-1}(q)$.}
\]

First let $K^{(\infty)}=q^{n-1}{:}\SL_{n-1}(q)$. Then $\Nor_L(K^{(\infty)})=\Pa_i$ is a maximal subgroup of $L$, where $i\in\{1,n-1\}$.
Since $K\leqslant\Pa_i[G]$ and $G=HK$, we have $G=H\Pa_i[G]$, and so $H$ is transitive on $\mathcal{P}_i[G]$.
By a result of Hering (see~\cite[Lemma~3.1]{LPS2010}), $H^{(\infty)}$ is one of the following groups:
\[
\text{$\lefthat\SL_a(q^b)$ with $n=ab$,\, $\lefthat\Sp_a(q^b)$ with $n=ab$\, or\, $\G_2(q^b)'$ with $n=6b$ and $q$ even.}
\]
However, for these groups, $|H^{(\infty)}|_{\pi'}|K^{(\infty)}|_{\pi'}$ does not divide $|L|_{\pi'}$, contradicting Lemma~\ref{LemXia4}.

Next assume that $K^{(\infty)}=\SL_{n-1}(q)$. Then $K^{(\infty)}$ stabilizes an antiflag of $\mathbb{F}_q^n$.
Noticing that $\Aut(L)=\PGaL_n(q).2$ acts primitively on the set of antiflags, we see that the stabilizer $\Nor_{\Aut(L)}(K^{(\infty)})$ of an antiflag is a maximal subgroup of $\Aut(L)$.
%Let $Y=L.2\leqslant\Aut(L)$ such that $Y\nleqslant\PGaL_n(q)$. Then $\Nor_Y(K^{(\infty)})$ is the stabilizer of an antiflag and is maximal in $Y$.
%Hence, in particular, $K$ is almost maximal in $G$ by definition.
Since $G=HK$ is transitive on the set of antiflags and $K$ is contained in the stabilizer of an antiflag in $G$, it follows that $H$ is antiflag-transitive.
Then by~\cite[Theorem~3.2]{LPS2010},
$H^{(\infty)}$ is one of the following groups:
\begin{align*}
&\text{$\lefthat\Sp_n(q)$,\,~$\lefthat\SL_{n/2}(q^2)$ with $q\in\{2,4\}$,\,~$\Sp_{n/2}(q^2)$ with $q\in\{2,4\}$,}\\
&\text{$\G_6(q)'$ with $n=6$ and $q$ even,\, or\, $\G_2(q^2)$ with $n=12$ and $q\in\{2,4\}$.}
\end{align*}
For these groups, $|H^{(\infty)}|_{\pi'}|K^{(\infty)}|_{\pi'}$ does not divide $|L|_{\pi'}$, contradicting Lemma~\ref{LemXia4}.
\end{proof}

\subsection{Unitary groups}
\ \vspace{1mm}

Let $G$ be a unitary group with socle $L=\PSU_n(q)$, where $n\geqslant3$, let $\sfM(L)$ be the Schur multiplier, and let $\pi=\pi(\sfM(L))$.
Then either $\sfM(L)=\ZZ_{(n,q+1)}$, or $(n,q,|\sfM(L)|)=(4,2,2)$, $(4,3,36)$ or $(6,2,12)$.
Thus either $\pi=\pi(\ZZ_{(n,q+1)})$ or $\pi\subseteq\{2,3\}$.
%We describe near-exact factorizations of $G$ in the following proposition.

\begin{proposition}\label{ThmUnitary}
Let $G=HK$ be a near-exact factorization with nonsolvable $H$ and $K$.
Then, interchanging $H$ and $K$ if necessary, $(G,H,K)$ lies in the following table, where $\calO_1\calO_2=\calO\leqslant\ZZ_2$ and $\ell=|\calO_1||\calO_2|/|\calO|$ $(=1$ or $2)$.
\[
\begin{array}{|llll|}
\hline
G & H & K & H\cap K \\
\hline
\PSU_4(3).(2^2.\calO) & \PSL_2(7).(2^2.\calO_1) & (3^4{:}\A_6).(2^2.\calO_2) & \D_{6\ell} \\
\PSU_4(3).(4.\calO) & \PSL_2(7).(4.\calO_1) & (3^4{:}\A_6).(4.\calO_2) & \D_{6\ell} \\
\SU_4(4).4 & \SL_2(16).4 & \SU_3(4).4 & 1 \\
\hline
\end{array}
\]
\end{proposition}

\begin{proof}
If $n$ is prime, then the proposition follows from~\cite[Theorem~3.5]{LX2022}.
If $L$ is one of
\[
\text{$\SU_4(2)$,\, $\PSU_4(3)$,\, $\SU_4(4)$,\, $\PSU_4(5)$,\, $\PSU_6(2)$,\, $\SU_8(2)$ or $\PSU_9(2)$,}
\]
then computation in \magma~\cite{BCP1997} leads to the candidates for $(G,H,K)$ in Proposition~\ref{ThmUnitary}.

Now assume $n\geqslant4$ and $L$ is not any of these seven groups.
%$(n,q)\notin\{(4,2),(4,3),(4,4),(4,5),(6,2),(8,2),(9,2)\}$.
%For the rest of the proof we will show that such a case is not possible.
By Lemmas~\ref{LemXia5}-\ref{LemXia27}, we may assume that $H\leqslant A$ and $K\leqslant B$ for some maximal subgroups $A$ and $B$ of $G$.
%We first assume that $(m,q)=(9,2)$.
%Then $A'=\SU_9(2)$.
%For this case, $B=\Pa_1[G]$ or $\N_1[G]$,
%Clearly, $G=AB$ is not a near-exact factorization.
%Note that $|G|\over|B|$ is divisible by 19,
%so is $|H|$.
%Since $A'=(3.\J_3)\Pa_1[A']$, we have $H=3.\J_3$,
%and hence $B=\Pa_1[G]$.
%By \cite[5.2.12]{LPS1990}, one has $H\cap B'=H\cap\Pa_1[A']=2^{2+4}.(3\times\Sy_3)$.
%It is easy to verify that $G=HK$ with $K=B'$
%is a minimal factorization.
The maximal factorizations $G=AB$ are classified in~\cite{LPS1990}, from which we obtain that $n=2m$, and $B^{(\infty)}=\SU_{2m-1}(q)$.
% and $A^{(\infty)}$ is one of the groups:
% \[
% \text{$\lefthat(q^{m^2}.\SL_m(q^2))$,\, $\PSp_{2m}(q)$,\, $\lefthat\SL_m(q^2)$ with $q\in\{2,4\}$,\, $\Suz$.}
% \]
By~\cite[Theorem~A]{LPS1990} we conclude that there is no factorization of a group with socle $\PSU_{2m-1}(q)$ with a nonsolvable factor. Then, since $B=(A\cap B)K$ with $K$ nonsolvable, it follows that $K^{(\infty)}=B^{(\infty)}=\SU_{2m-1}(q)$.

Notice that $G$ is transitive on $\mathcal{N}_1[G]$, the set of nonsingular $1$-spaces.
Since $K\leqslant B=\N_1[G]$, we have $G=HK=H\N_1[G]$, and so $H$ is transitive on $\mathcal{N}_1[G]$.
Then by~\cite[Lemma~4.3]{LPS2010}, one of the following occurs:
\begin{enumerate}[{\rm (i)}]
\item $H^{(\infty)}=\lefthat\Sp_{2m}(q)'$, $\Sp_m(q^2)$ ($q\in\{2,4\}$) or $\Sp_{m/2}(q^4)$ ($q=2$);
\item $H^{(\infty)}=\lefthat\SL_m(q^2)$ ($q\in\{2,4\}$) or $\lefthat\SL_{m/2}(q^4)$ ($q=2$);
\item $H^{(\infty)}=\G_2(q)'$ ($m=3$ and $q$ even), $\G_2(q^2)$ ($m=6$ and $q\in\{2,4\}$) or $\G_2(q^4)$ ($m=12$ and $q=2$);
%\item $m=3$, $q=2$, and $H^{(\infty)}=\PSU_4(3)$ or $\M_{22}$, as in row~6 of Table~\ref{TabUnitary};
\item $L=\PSU_{12}(2)$ and $H^{(\infty)}=\Suz$;
\item $H\cap L=\lefthat(P.R)$, where $P\leqslant q^{m^2}$ and $R\leqslant\GaL_m(q^2)$ is transitive on $\mathcal{P}_1[\GaL_m(q^2)]$.
\end{enumerate}
%If~(i) occurs, then $(L,H^{(\infty)},K^{(\infty)})$ lies in row~1 or~2 of~\ref{TabUnitary}. If~(ii) occurs, then $(L,H^{(\infty)},K^{(\infty)})$ lies in row~1 of~\ref{TabUnitary}. If~(iii) occurs, then $(L,H^{(\infty)},K^{(\infty)})$ lies in row~1 or~3 of~\ref{TabUnitary}.
Moreover, by~\cite[Lemma~3.1]{LPS2010}, the group $R$ in~(v) satisfies one of the following:
\begin{itemize}
\item[(v.1)] $R^{(\infty)}=\SL_a(q^{2b})$ with $m=ab$;
\item[(v.2)] $R^{(\infty)}=\Sp_a(q^{2b})$ with $m=ab$;
\item[(v.3)] $R^{(\infty)}=\G_2(q^{2b})$ with $m=6b$ and $q$ even.
\end{itemize}
Since $H\cap K$ is a $\pi$-group, where $\pi=\pi(\sfM(L))=\pi(\ZZ_{(n,q+1)})$, Lemma~\ref{LemXia4} asserts that $|H^{(\infty)}|_{\pi'}|K^{(\infty)}|_{\pi'}$ divides $|L|_{\pi'}$.
It follows that $|H^{(\infty)}|_{\pi'}$ divides \[\mbox{$|L|_{\pi'}/|K^{(\infty)}|_{\pi'}=|\PSU_{2m}(q)|_{\pi'}/|\SU_{2m-1}(q)|_{\pi'}
=q^{2m-1}(q^{2m}-1)_{\pi'}$.}\]
Inspecting the candidates for $H^{(\infty)}$ in cases (i)--(v), we conclude that $H\cap L=\lefthat(P.R)$ with $P\leqslant q^{m^2}$ and $\SL_2(q^m)\trianglelefteq R\leqslant\GaL_m(q^2)$, as in case~(ii) or~(v.1).
Then $P=\lefthat P$ is a normal subgroup of $H$, and $\SL_2(q^m)\trianglelefteq H/P\leqslant\GaL_2(q^m)$.
Let $q=p^f$ with prime $p$ and integer $f$. Note that $|K\cap L|_p=|\SU_{2m-1}(q)|_p$ and so
\begin{equation}\label{EqnXia1}
|P||H/P|_p=|H|_p=\frac{|G|_p}{|K|_p}=\frac{|L|_p}{|K\cap L|_p}=\frac{|\SU_n(q)|_p}{|\SU_{n-1}(q)|_p}=q^{n-1}=q^{2m-1}.
\end{equation}
Since $|H/P|_p\leqslant|\GaL_2(q^m)|_p\leqslant fmq^m$, this implies that
\[
|P|=\frac{q^{2m-1}}{|H/P|_p}\geqslant\frac{q^{2m-1}}{fmq^m}=\frac{q^{m-1}}{fm}>1.
\]
Moreover, since $|H/P|_p\geqslant|\SL_2(q^m)|_p=q^m$, it follows from~\eqref{EqnXia1} that
\[
|P|=\frac{q^{2m-1}}{|H/P|_p}\leqslant\frac{q^{2m-1}}{q^m}=q^{m-1}.
\]
Let $r\in\ppd(q^{2m}-1)$.
Then $r$ divides $|P.R|$ as $\SL_2(q^m)\trianglelefteq R$. On the other hand, $r$ does not divide $|\Aut(P)|$ since $P$ is a $p$-group of order at most $q^{m-1}$. Hence $r$ divides $|\Cen_{P.R}(P)|$ and thus divides $|\Cen_{(P.R)\cap\GU_{2m}(q)}(P)|$. Take $1\neq u\in P$. Then $u$ is a unipotent element in $\GU_{2m}(q)$, and
\[
\Cen_{(P.R)\cap\GU_{2m}(q)}(P)\leqslant\Cen_{(P.R)\cap\GU_{2m}(q)}(u)\leqslant\Cen_{\GU_{2m}(q)}(u).
\]
Let $\bigoplus_iJ_i^{n_i}$ be the Jordan canonical form of $u$, where $J_i$ is a unipotent Jordan block of length $i$ and $\sum_iin_i=2m$. From~\cite[Theorem~7.1(ii)]{LS2012} we see that $|\Cen_{\GU_{2m}(q)}(u)|_{p'}$ divides $\prod_i|\GU_{n_i}(q)|$. Consequently, $r$ divides $\prod_i|\GU_{n_i}(q)|$. Since $r$ is a primitive prime divisor of $q^{2m}-1$, this yields $i=1$ and $n_1=2m$. Hence the Jordan canonical form of $u$ is $\bigoplus_iJ_i^{n_i}=J_1^{2m}$, which implies that $u=1$, which a contradiction.
\end{proof}

\subsection{Orthogonal groups in odd dimension}
\ \vspace{1mm}

Let $L=\Soc(G)=\Omega_{2m+1}(q)$ with $m\geqslant3$ and $q$ odd, and let $\pi=\pi(\sfM(L))$.
Then either $\sfM(L)=\ZZ_2$ or $(L,\sfM(L))=(\Omega_7(3),\ZZ_6)$, so that either $\pi=\{2\}$ or $\pi=\{2,3\}$ for $L=\Omega_7(3)$.

%In this subsection we prove Theorem~\ref{ThmNearExact} for orthogonal groups.

\begin{proposition}\label{ThmOmega}
Let $G=HK$ be a near-exact factorization with nonsolvable $H$ and $K$.
Then $G=\Omega_7(3)$ or $\Omega_7(3).2$.
\end{proposition}

\begin{proof}
%For $L=\Omega_7(3)$, computation in \magma~\cite{BCP1997} shows that the near-exact factorization $G=HK$ lies in the table of Proposition~\ref{ThmOmega}.

Suppose $(m,q)\neq(3,3)$.
Then $H\cap K$ is a 2-group.
By Lemmas~\ref{LemXia5} and~\ref{LemXia27}, we may assume that $G=HL=KL$ and there exist maximal subgroups $A$ and $B$ of $G$ containing $H$ and $K$ respectively.
The maximal factorizations $G=AB$ are classified in \cite{LPS1990}, listed in Table~\ref{TabMaxOmega}.
{\begin{table}[htbp]
\caption{Maximal factorizations of $G$ with $(m,q)\neq(3,3)$}\label{TabMaxOmega}
\centering
\begin{tabular}{|l|l|l|}
\hline
 $A\cap L$ & $B\cap L$ & Remark\\
\hline
 $\Pa_m$ & $\N_1^-$ & \\
 $\Pa_1$, $\N_1^+$, $\N_1^-$, $\N_2^+$, $\N_2^-$ & $\G_2(q)$ & $m=3$ \\
 $\PSp_6(q).a$ & $\N_1^-$ & $m=6$, $q=3^f$, $a\leqslant2$ \\
 $\F_4(q)$ & $\N_1^-$ & $m=12$, $q=3^f$ \\
\hline
%5 & $\Sy_9$, $\Sp_6(2)$ & $\G_2(3)$ & $m=3$, $q=3$ \\
%6 & $\Sy_9$, $\Sp_6(2)$ & $\N_1^+$ & $m=3$, $q=3$ \\
%7 & $\Sy_9$, $\Sp_6(2)$, $2^6.\A_7$ & $\Pa_3$ & $m=3$, $q=3$ \\
%\hline
\end{tabular}
\end{table}}
Notice that $B\cap L=\N_1^-$, or $B\cap L=\G_2(q)$ with $m=3$.

{\sc Case 1.}\, Let $B\cap L=\G_2(q)$, so that $L=\Omega_7(q)$.
Then $B$ is an almost simple group with socle $\Omega_7(q)$.
Consulting~\cite[Theorem~1.1]{HLS1987}, we see that $\G_2(q)$ has no near-exact factorization, and hence $K^{(\infty)}=B^{(\infty)}=\G_2(q)$.
As $G=KL$, this implies that $K=B$ is maximal in $G$.
By~\cite[5.1.15]{LPS1990}, there is a factorization $X=GY$ of an almost simple group $X$ with socle $\POm_8^+(q)$ such that $Y=\N_1[X]$ and $G\cap Y=K$.
Thus it follows from $G=HK$ that $X=HY$.
Since Lemma~\ref{LemXia4} implies that $|H^{(\infty)}|_{2'}$ divides ${|G^{(\infty)}|_{2'}/|K^{(\infty)}|_{2'}}=q^3(q^4-1)_{2'}$, we derive from~\cite[Lemma~4.5]{LPS2010} that $H<M$ for some stabilizer $M$ of a totally singular $4$-space in $X$ and $H^{(\infty)}=\PSL_2(q^2)$.
Let $S=H^{(\infty)}\cong\Omega_4^-(q)$, let $R$ be the unipotent radical of the parabolic subgroup $M$ of $X$, and let $V_2$ be a natural module for $H^{(\infty)}$.
By \cite[Corollarly~4.5]{AJL1983}, we have ${\rm Ext}_{S}(V_2\otimes V_2^{(q)},{\rm triv}_2)=0$, and hence $R\downarrow S=V_2\otimes V_2^{(q)}\oplus {\rm triv}_2$ (realized over $\mathbb{F}_q$) (\cite[Page~38]{LPS2010}), where ${\rm triv}_2$ is a $2$-dimensional trivial module.
Moreover, we deduce from~\cite[Corollarly~4.5]{AJL1983} that ${\rm H^1}(S,V_2\otimes V_2^{(q)})=0$, and so the semidirect product $RS$ has a unique conjugacy class of complements to $R$. Thus we may assume $S\leqslant\N_3[G]$. It follows that $H\leqslant\N_3[G]$.
However, this leads to $G=\N_3[G]B$, contradicting~\cite[Theorem~A]{LPS1990}.

{\sc Case 2.}\, Let $B\cap L=\N_1^-$.
Then $B$ is an almost simple group with socle $B^{(\infty)}=\Omega_{2m}^-(q)$, and $A^{(\infty)}$ is one of the following:
\[\mbox{$(q^{m(m-1)/2}.q^m){:}\SL_m(q)$,
$\PSp_6(q)$ with $(m,q)=(6,3^f)$, $\F_4(q)$ with $(m,q)=(12,3^f)$}.\]
%Then the index ${|L|\over|B^{(\infty)}|}={|\Omega_{2m+1}(q)|\over|\POm_{2m}^-(q)|}=q^m(q^m-1)$.
 %is divisible by
%\[r\in\ppd(q^{2m}-1)\ \text{ and }\ s\in\ppd(q^{2m-2}-1),\]
%and so is the order $|H|$.

Suppose that $K^{(\infty)}<B^{(\infty)}$.
Since the Schur multiplier $\sfM(B^{(\infty)})=\sfM(\Omega_{2m}^-(q))=\ZZ_{(4,q^m+1)}$, it follows that $B=(H\cap B)K$ is a near-exact factorization, and so satisfies Theorem~\ref{ThmNearExact} as $G=HK$ is a minimal counterexample.
This is not possible by \cite[Theorem~1.1]{LX2022} and Theorem~\ref{ThmNearExact}.

We thus conclude that $K^{(\infty)}=B^{(\infty)}=\Omega_{2m}^-(q)$. Then Lemma~\ref{LemXia4} implies that $|H^{(\infty)}|_{2'}$ divides ${|G^{(\infty)}|_{2'}/|K^{(\infty)}|_{2'}}=q^m(q^m-1)_{2'}$. Now consider the factorization $A=H(A\cap K)$.
If $A^{(\infty)}=\F_4(q)$, then by~\cite[Theorem~1]{HLS1987} we have $H^{(\infty)}=A^{(\infty)}=\F_4(q)$, contradicting the conclusion that $|H^{(\infty)}|_{2'}$ divides $q^m(q^m-1)_{2'}$.
If $A^{(\infty)}=\PSp_6(q)$ with $(m,q)=(6,3^f)$, then since $G=HK$ is a minimal counterexample, we conclude from Theorem~\ref{ThmNearExact}
and~\cite[Theorem~1.1]{LX2022} that $q=3$ and $H^{(\infty)}=\PSL_2(27)$, again contradicting the conclusion that $|H^{(\infty)}|_{2'}$ divides $q^m(q^m-1)_{2'}$.

Thus $A^{(\infty)}=(q^{m(m-1)/2}.q^m){:}\SL_m(q)$. Let $\overline{\phantom{x}}\colon A\to A/\Rad(A)$ be the quotient modulo $\Rad(A)$.
Then $\ov A$ is almost simple with socle $\PSL_m(q)$, and $\overline{A}=\overline{H}\,\overline{(A\cap K)}$.
Since $q$ is odd and $|\overline{H}|$ is divisible by $\ppd(q^m-1)$, it follows from Table~\ref{TabMaxLinear} that $\ov{A\cap K}$ stabilizes a $1$-space or an antiflag, and so $\ov H$ is transitive on the set of $1$-spaces or the set of antiflags.
Since $H$ is nonsolvable with $|H^{(\infty)}|_{2'}$ dividing $q^m(q^m-1)_{2'}$, we conclude from~\cite[Lemma~3.1 and Theorem~3.2]{LPS2010} that one of the following holds:
\begin{itemize}
\item[(i)]
$H^{(\infty)}=P.\SL_a(q^b)$ or $P.\Sp_a(q^b)$ with $m=ab$ and $P\leqslant q^{m(m-1)/2}.q^m$;
\item[(ii)]
$H^{(\infty)}=P.2^{1+4}.\A_5$ or $P.\SL_2(5)$ with $(m,q)=(4,3)$ and $P\leqslant 3^{6+4}$;
\item[(iii)]
$H^{(\infty)}=P.\SL_2(13)$ with $(m,q)=(6,3)$ and $P\leqslant 3^{15+6}$.
\end{itemize}
The candidates in case~(ii) are excluded by computation in \magma~\cite{BCP1997}.

Next assume that~(i) occurs. Then $\SL_2(q^b)\trianglelefteq R\leqslant\GL_2(q^b).b$. Since $|H\cap L|_p=|P||R|_p$ and $|K\cap L|_p=|\Omega_{2m}^-(q)|_p$, it follows from Lemma~\ref{LemXia4} that
\begin{equation}\label{EqnXia5}
|G/L|_p|P||R|_p=|G/L|_p|H\cap L|_p=\frac{|L|_p}{|K\cap L|_p}=\frac{|\Omega_{2m+1}(q)|_p}{|\Omega_{2m}^-(q)|_p}=q^m=q^{2b}.
\end{equation}
As $|G/L|_p\leqslant f$ and $|R|_p\leqslant|\GL_2(q^b).b|_p\leqslant q^bb$, this implies that
\[
|P|=\frac{q^{2b}}{|G/L|_p|R|_p}\geqslant\frac{q^{2b}}{f\cdot q^bb}=\frac{q^b}{fb}>1.
\]
Moreover, since $|R|_p\geqslant|\SL_2(q^b)|_p=q^b$, we deduce from~\eqref{EqnXia5} that
\[
|P|=\frac{q^{2b}}{|G/L|_p|R|_p}\leqslant\frac{q^{2b}}{q^b}=q^b.
\]
Let $r\in\ppd(q^{2b}-1)=\ppd(q^m-1)$. Then $r$ divides $|P.R|$, but does not divide $|\Aut(P)|$ as $|P|\leqslant q^b<q^m$. Hence $r$ divides $|\Cen_{P.R}(P)|$, and thus divides $|\Cen_{\Pa_m[L]}(P)|$ as $P.R\leqslant\Pa_m[L]$. Let $T$ be a subgroup of order $r$ in $\Cen_{\Pa_m[L]}(P)$, and let $Q$ be the unipotent radical of $\Pa_m[L]$. Then $\Cen_Q(T)=1$ (see~\cite[Page~32,~Line~2]{LPS2010}), and so $T$ does not centralize any nontrivial element in $Q$. In particular, $T$ does not centralize $P$, contradicting our choice of $T$.

Finally, assume that~(iii) occurs. Then $\SL_2(13)\trianglelefteq R<\GL_6(3)$. Since $\Nor_{\GL_6(3)}(\SL_2(13))=\SL_2(13)$, it follows that $R=\SL_2(13)$. In particular, $|R|_3=|\SL_2(13)|_3=3$. This in conjunction with Lemma~\ref{LemXia4} implies that
\[
3|P|=|G/L|_3|P||R|_3=|G/L|_3|H\cap L|_3=\frac{|L|_3}{|K\cap L|_3}=\frac{|\Omega_{13}(3)|_3}{|\Omega_{12}^-(3)|_3}=3^6,
\]
and so $|P|=3^5<q^m$. Then the similar argument as in case~(i) leads to a contradiction.
\end{proof}

\subsection{Orthogonal groups of minus type}
\ \vspace{1mm}

Let $L=\Soc(G)=\POm_{2m}^-(q)$ with $m\geqslant4$.
Then $\sfM(L)=\ZZ_{(4,q^m+1)}$ and so $\pi\subseteq\{2\}$.

%In this subsection we prove Theorem~\ref{ThmNearExact} for orthogonal groups of minus type.

\begin{proposition}\label{ThmOmegaMinus}
There are no nonsolvable subgroups $H,K$ of $G$ such that
$G=HK$ is a near-exact factorization.
%subject to the condition where $H\cap K$ is a $\pi$-group.
\end{proposition}

Suppose $G=HK$ with nonsolvable $H$ and $K$ such that $H\cap K$ is a $\pi$-group.
By Lemmas~\ref{LemXia5} and~\ref{LemXia27}, we may assume that $G=HL=KL$ and $H\leqslant A$ and $K\leqslant B$ for some maximal subgroups $A$ and $B$ of $G$.
Then $G=AB$ is given in \cite{LPS1990}, see Table~\ref{TabMaxOmegaMinus}.

{\begin{table}[htbp]
\caption{Maximal factorizations of $G$}\label{TabMaxOmegaMinus}
\centering
\begin{tabular}{|l|l|l|l|}
\hline
 & $A\cap L$ & $B\cap L$ & Remark\\
\hline
1 & $\lefthat\GU_m(q)$ & $\Pa_1$ & $m$ odd \\
2 & $\lefthat\GU_m(q)$ & $\N_1$ & $m$ odd \\
3 & $\Omega_m^-(q^2).2$ & $\N_1$ & $m$ even, $q\in\{2,4\}$, $G=\Aut(L)$ \\
4 & $\GU_m(4)$ & $\N_2^+$ & $m$ odd, $q=4$, $G=\Aut(L)$ \\
5 & $\A_{12}$ & $\Pa_1$ & $m=5$, $q=2$ \\
\hline
\end{tabular}
\end{table}}

Computation in \magma~\cite{BCP1997} shows that there is no near-exact factorization for $L=\POm_{10}^-(2)$ with $(A\cap L,B\cap L)=(\A_{12},\Pa_1)$, excluding the candidate in row~5.
%\[
%(H\cap L,K\cap L)\in\{(\SU_5(2).a,\Pa_1),\,(\SU_5(2).a,\Omega_8^-(2)),\,(\A_{12},\Pa_1),\,(\M_{12},\Pa_1)\},
%\]
%This is not possible since now $|H^{(\infty)}||K^{(\infty)}|$ should divide $|G^{(\infty)}|$.

\begin{lemma}\label{lem:O-minus-3}
Let $A^{(\infty)}=\Omega_m^-(q^2)$, and $B^{(\infty)}=\Omega_{2m-1}(q)$ with $m$ even and $q\in\{2,4\}$.
Then $G=HK$ is not a near-exact factorization.
\end{lemma}

\begin{proof}
As $q\in\{2,4\}$, the Schur multiplier $\sfM(L)=1$, and so $H\cap K=1$.
For $m\in\{4,6\}$, computation in \magma~\cite{BCP1997} shows that there is no such factorization $G=HK$. Now assume $m\geqslant8$.
By~\cite[Prop.4.3.16]{KL1990}, the group $A$ is almost simple with socle $\Omega_m^-(q^2)$.
Under Hypothesis~\ref{hypo}, Theorem~\ref{ThmNearExact} implies that $A$ has no near-exact factorization.
Since $A=H(A\cap K)$, we conclude that $H^{(\infty)}=A^{(\infty)}=\Omega_m^-(q^2)$.
As $|L|/|A\cap L|$ is divisible by
\[
r\in\ppd(q^{2m-2}-1),\, s\in\ppd(q^{2m-6}-1)\,\, \text{and }\,\, t\in\ppd(q^{m-1}-1),
\]
so is $|K|$. Note that $B=(H\cap B)K$, and by~\cite[Theorem~A]{LPS1990}, no almost simple group with socle $\Omega_{2m-1}(q)\cong\Sp_{2m-2}(q)$ has a nontrivial factor of order divisible by $rst$. We have $K^{(\infty)}=B^{(\infty)}=\Omega_{2m-1}(q)$.
However, this implies that $|H||K|\geqslant|\Omega_m^-(q^2)||\Omega_{2m-1}(q)|>|G|$, which is a contradiction to the condition $H\cap K=1$.
\end{proof}

%It is easy to see that $|H^{(\infty)}|_{2'}|B^{(\infty)}|_{2'}$ does not divide $|L^{(\infty)}|$, and hence $B=KB_\o$ is a near-exact factorization.
%Note that $|K^{(\infty)}|$ is divisible by $r$.
%By Theorem~\ref{ThmNearExact},
%either $B_\o$ is solvable, or
%$(B^{(\infty)},B_\o^{(\infty)},K^{(\infty)})=(\Sp_6(4),\SL_2(16),\G_2(4))$.
%For the latter, $|K|$ is indivisible by ${|L|\over|H^{(\infty)}|}$,
%which is a contradiction.
%For the former, \cite[Theorem~1.1]{LX} tells us that $K^{(\infty)}=\Omega_{2m-2}^-(q)$.
%However, $|H^{(\infty)}|_{2'}|K^{(\infty)}|_{2'}$ does not divide $|L|$, which is not  possible.

%$K^{(\infty)}=\Omega_{2m-2}^-(q)$, or $(m,q,K^{(\infty)})=(3,3,\G_2(3))$, $(3,3,\Sp_6(2))$ or $(m,3,\Omega_8^-(3))$.
%Then $|H^{(\infty)}|_{2'}|K^{(\infty)}|_{2'}$ does not divide $|L|$, which is not  possible.
%
%Moreover, we derive from $G=HB=H\N_1[G]$ that $H$ is transitive on $\mathcal{N}_1[G]$.
%By \cite[Lemma~4.4]{LPS2010} with $m\geqslant4$, $H^{(\infty)}$ is one of the groups:
%\begin{quote}
%$\SU_{m/2}(q^2)$ with $m/2$ odd, or $\SU_{m/4}(q^4)$ with $m/4$ odd and $q=2$, or
%$\Omega_m^-(q^2)$, or $\Omega_{m/2}^-(q^4)$ with $q=2$.
%\end{quote}
%Calculation shows that $|H^{(\infty)}||K^{(\infty)}|$ divides $|G^{(\infty)}|$, a contradiction.

We now consider the case $A\cap L=\lefthat\GU_m(q)$, as in rows~1, 2 and 4 of Table~\ref{TabMaxOmegaMinus}.

\begin{lemma}\label{lem:Unitary-1}
For $A^{(\infty)}=\SU_m(q)$ with $m$ odd, $G=HK$ is not a near-exact factorization.
\end{lemma}

\begin{proof}
Let $M=\Rad(A)$. Then $A/M$ is an almost simple group with socle $\PSU_m(q)$.

{\sc Case 1.}\, Assume that $B\cap L=\Pa_1$, as in row~1 of Table~\ref{TabMaxOmegaMinus}. Since
\[
B^{(\infty)}=\lefthat\Big(q^{2m-2}{:}\Omega^-_{2m-2}(q)\Big)=q^{2m-2}{:}\Omega^-_{2m-2}(q),
\]
it follows that $|L|/|A\cap L|$ is divisible by
\[
r\in\ppd(q^{2m-2}-1)\ \text{ and }\ s\in\ppd(q^{m-2}-1),
\]
and so is $|K|$. Moreover, $|L|/|B\cap L|$ is divisible by $t\in\ppd(q^{2m}-1)$, and so is $|H|$.
Let $N=\Rad(B)$. Then $B/N$ is a group with socle $\POm_{2m-2}^-(q)$.
Since $|K|$ is divisible by both $r$ and $s$, so is $|KN/N|$.
Inspecting the maximal factorizations of groups with socle $\POm_{2m-2}^-(q)$, we see that there is no core-free factor of $B/N$ with order divisible by both $r$ and $s$. Thus $KN/N\trianglerighteq\Soc(B/N)=\POm_{2m-2}^-(q)$.
In particular, $|K^{(\infty)}|_{\pi'}$ is divisible by $|\POm_{2m-2}^-(q)|_{\pi'}$.
Since Lemma~\ref{LemXia4} shows that $|H^{(\infty)}|_{\pi'}|K^{(\infty)}|_{\pi'}$ divides $|\POm_{2m}^-(q)|_{\pi'}$, it follows that
\begin{equation}\label{EqnXia7}
|H^{(\infty)}|_{\pi'}\text{ divides }q^{2m-2}(q^m+1)(q^{m-1}-1).
\end{equation}
We now consider $H$ in its overgroup $A$.
Since $|H|$ is divisible by $t\in\ppd(q^{2m}-1)$, so is $|HM/M|$.
If $HM/M\trianglerighteq\Soc(A/M)=\PSU_m(q)$, then~\eqref{EqnXia7} would not hold.
Hence $HM/M$ is a core-free factor of $A/M$ with order divisible by $t$.
Since $m\geqslant5$ is odd, we see from~\cite[Theorem~A]{LPS1990} that $(m,q)=(9,2)$ and $(HM/M)\cap\Soc(A/M)\leqslant\J_3$.
Moreover, searching for such factors $HM/M$ in $A/M$ in \magma~\cite{BCP1997} gives $(HM/M)\cap\Soc(A/M)=\J_3$.
Thus $H\cap L=3.\J_3$ is a maximal subgroup of $A^{(\infty)}=\SU_9(2)$ (see~\cite[Table~8.57]{BHR2013}), and so $|H|_2\leqslant|H\cap L|_2|\Out(L)|_2=2|\J_3|_2=2^8$. Since $|H|_2|K|_2\geqslant|G|_2\geqslant|L|_2=2^{72}$, this yields $|K|_2\geqslant2^{64}$, which together with $KN/N\trianglerighteq\POm_{2m-2}^-(q)$ implies $K^{(\infty)}=2^{16}{:}\Omega_{16}^-(2)$. However, it follows that $|H^{(\infty)}|_{\pi'}|K^{(\infty)}|_{\pi'}$ does not divide $|\POm_{2m}^-(q)|_{\pi'}$, a contradiction.

{\sc Case 2.}\, Assume that $B\cap L=\N_1$, as in row~~2 of Table~\ref{TabMaxOmegaMinus}.
Let $N=\Rad(B)$, and let $\overline{\phantom{x}}\colon B\to B/N$ be the quotient modulo $N$. Then $\overline{B}$ is an almost simple group with socle $\Omega_{2m-1}(q)$, and $\overline{K}$ is a factor of $\overline{B}$.
Since $A\cap B\cap L\leqslant\N_1[A\cap L]$ (see~\cite[3.5.2(b)]{LPS1990}), we have $(A\cap B)M/M\leqslant\N_1[A/M]$.
Since $A=H(A\cap B)$, the group $HM/M$ is a supplement of $(A\cap B)M/M$ in $A/M$ and hence is a supplement of $\N_1[A/M]$ in $A/M$.
As $m\geqslant5$ is odd, we see from~\cite[Theorem~A]{LPS1990} that there is no core-free supplement of $\N_1[A/M]$ in $A/M$.
Thus $HM/M\trianglerighteq\Soc(A/M)=\PSU_m(q)$, and so $H^{(\infty)}=A^{(\infty)}=\SU_m(q)$.

Since $|L|/|A\cap L|=|\POm_{2m}^-(q)|/|\lefthat\GU_m(q)|$ is divisible by
\[
r\in\ppd(q^{2m-2}-1),\ \ s\in\ppd(q^{2m-6}-1)\ \text{and }\ t\in\ppd(q^{m-2}-1),
\]
the order $|K|$ is divisible by $rst$, and so is $|\overline{K}|$. Since $\overline{K}$ is a factor of $\overline{B}$ with $\Soc(\overline{B})=\Omega_{2m-1}(q)$, \cite[Theorem~A]{LPS1990} implies that one of the following appears:
\begin{enumerate}[{\rm (i)}]
\item $\overline{K}\trianglerighteq\Soc(\overline{B})=\Omega_{2m-1}(q)$;
\item $\overline{K}\cap\Soc(\overline{B})\leqslant\Omega_{2m-2}^-(q).2$.
\end{enumerate}
For~(i), $|K^{(\infty)}|$ is divisible by $|\Omega_{2m-1}(q)|$, contradicting the conclusion from Lemma~\ref{LemXia4} that $|H^{(\infty)}|_{\pi'}|K^{(\infty)}|_{\pi'}$ divides $|L|_{\pi'}$.

Thus~(ii) appears. Let $\overline{B}_1$ be a maximal core-free subgroup of $\overline{B}$ containing $\overline{K}$ such that $\overline{B}_1\cap\Soc(\overline{B})=\Omega_{2m-2}^-(q).2$, and let $B_1$ be a subgroup of $B$ containing $N$ such that $B_1/N=\overline{B}_1$. Then $\overline{B}_1$ is an almost simple group with socle $\Omega_{2m-2}^-(q)\cong\POm_{2m-2}^-(q)$ (as $m$ is odd), and $K\leqslant B_1$. Thus it follows from $G=AK$ that $B_1=(A\cap B_1)K$ and so
\[
\overline{B}_1=(\overline{A\cap B_1})\overline{K}.
\]
Since $A\cap B_1\cap L\leqslant A\cap B\cap L\leqslant\N_1[A\cap L]$, we see that $\overline{A\cap B_1}$ is core-free in $\overline{B}_1$. Then as $|\overline{K}|$ is divisible by $rst$, we conclude from~\cite[Theorem~A]{LPS1990} that $\overline{K}\trianglerighteq\Soc(\overline{B}_1)=\Omega_{2m-2}^-(q)$, again contradicting Lemma~\ref{LemXia4}.

{\sc Case 3.}\, Let $B\cap L=\N_2^+$, as in row~4 of Table~\ref{TabMaxOmegaMinus}.

In this case, $q=4$, $A^{(\infty)}=\SU_m(4)$, and $B^{(\infty)}=\Omega_{2m-2}^-(4)$. Since $|L|/|A\cap L|$ is divisible by $r\in\ppd(4^{2m-2}-1)$ and $s\in\ppd(4^{m-2}-1)$, so is $|K|$.

Let $N=\Rad(B)$. Then the group $B/N$ is almost simple with socle $\Omega_{2m-2}^-(4)$, and $KN/N$ is a factor of $B/N$ with order divisible by $r$ and $s$.
%we see that $KN/N$ is a factor of $B/N$ with $|KN|$ divisible by $|L|/|A\cap L|$. As $|L|/|A\cap L|$ is divisible by $\ppd(4^{2m-2}-1)$ and $\ppd(4^{m-2}-1)$, we then conclude by Lemma~\ref{LemXia7}\,(2) that $|KN/N|$ is divisible by $\ppd(4^{2m-2}-1)$ and $\ppd(4^{m-2}-1)$.
By the classification in~\cite{LPS1990} of the maximal factorizations of almost simple groups with socle $\Omega_{2m-2}^-(4)$, there is no core-free factor of $B/N$ with order divisible by both $r\in\ppd(4^{2m-2}-1)$ and $s\in\ppd(4^{m-2}-1)$. Hence $KN/N\trianglerighteq\Soc(B/N)=\Omega_{2m-2}^-(4)$, and so $K^{(\infty)}=B^{(\infty)}=\Omega^-_{2m-2}(4)$.

Let $M=\Rad(A)$. Then $A/M$ is an almost simple group with socle $\PSU_m(4)$, and $HM/M$ is a factor of $A/M$. According to the classification in~\cite{LPS1990} of the maximal factorizations of almost simple unitary groups $A/M$ has no factorization as $m$ is odd. Therefore, $HM/M\trianglerighteq\Soc(A/M)=\PSU_m(4)$, and so $H^{(\infty)}=A^{(\infty)}=\SU_m(4)$.
However, this contradicts Lemma~\ref{LemXia4} again.
\end{proof}

\subsection{Orthogonal groups of plus type}\label{SecXia2}
\ \vspace{1mm}

Let $L=\Soc(G)=\POm_{2m}^+(q)$, where $m\geqslant4$, and let $\pi=\pi(\sfM(L))$.
Then either $\pi=\pi(\ZZ_{(2,q-1)})$, or $\pi=\{2\}$ for $(m,q)=(4,2)$.
In particular, $\sfM(L)$ is a 2-group.

\begin{proposition}\label{ThmOmegaPlus}
Let $G=HK$ be a near-exact factorization with nonsolvable $H$ and $K$.
Then $(G,H,K)$ is one of the following, where $\calO_1,\calO_2\leqslant\calO\leqslant\mathrm{C}_2$ with $|\calO|=|\calO_1||\calO_2|$.
\[
\begin{array}{|llll|}
\hline
G & H & K & H\cap K \\
\hline
\Omega_8^+(2).\calO & 2^4.\A_5\leqslant H\leqslant 2^{1+6}.\Sy_5 & \A_9 & H\cap K\leqslant\D_8 \\
\Omega_8^+(2).\calO & 2^4.\A_5\leqslant H\leqslant 2^6.\Sy_5 & \A_9,\Sy_9 & H\cap K\leqslant\D_8 \\
\Omega_8^+(2).\calO & \A_5\leqslant H\leqslant 2^{1+6}.\Sy_5 & \Sp_6(2) & H\cap K\leqslant\D_8{:}\D_8 \\
\Omega_8^+(2).\calO & \A_5\leqslant H\leqslant 2^6.\Sy_5 & \Sp_6(2),\Sp_6(2)\times2 & H\cap K\leqslant\D_8{:}\D_8 \\ \hline
\Omega_8^+(4).(2\times\calO) & (\SL_2(16).4)\times\calO_1 & (\Sp_6(4).2)\times\calO_2 & 1\\
\hline
\end{array}
\]
\end{proposition}

To prove Proposition~\ref{ThmOmegaPlus}, by Lemmas~\ref{LemXia5} and~\ref{LemXia27}, we may assume that $G=HL=KL$ and there exist maximal subgroups $A$ and $B$ of $G$ containing $H$ and $K$ respectively.
We first treat a few small groups by computation in \magma~\cite{BCP1997}.

\begin{lemma}\label{LemXia29}
If $L=\Omega_8^+(2)$, $\POm_8^+(3)$ or $\Omega_8^+(4)$, then $L=\Omega_8^+(2)$ or $\Omega_8^+(4)$, and, interchanging $H$ and $K$ if necessary, $(L,H^{(\infty)},K^{(\infty)})$ lies in Proposition~$\ref{ThmOmegaPlus}$.
\end{lemma}

We assume next that $(m,q)\neq(4,2),\ (4,3)\,\text{ or }\,(4,4)$.
The maximal factorizations $G=AB$ are classified in~\cite{LPS1990}, as in Tables~\ref{TabMaxOmegaPlus2}\footnote{See~\cite[Remark~1.5]{GGP} for a correction of this table.} and~\ref{TabMaxOmegaPlus1}.

{\begin{table}[htbp]
\caption{Maximal factorizations of $G$ for $L=\POm_8^+(q)$ with $q\geqslant5$}\label{TabMaxOmegaPlus2}
\centering
\begin{tabular}{|l|l|l|l|}
\hline
 & $A\cap L$ & $B\cap L$ & Remark\\
\hline
1 & $\Omega_7(q)$ & $\Omega_7(q)$ & $A=B^\tau$ for some triality $\tau$ \\
2 & $\Pa_1$, $\Pa_3$, $\Pa_4$ & $\Omega_7(q)$ & \\
3 & $\lefthat((q+1)/d\times\Omega_6^-(q)).2^d$ & $\Omega_7(q)$ & $d=(2,q-1)$, $A$ in $\mathcal{C}_1$ or $\mathcal{C}_3$ \\
4 & $\lefthat((q-1)/d\times\Omega_6^+(q)).2^d$ & $\Omega_7(q)$ & $d=(2,q-1)$, $A$ in $\mathcal{C}_1$ or $\mathcal{C}_2$, \\
 & & & $q>2$, $G=L.2$ if $q=3$ \\
5 & $(\PSp_2(q)\otimes\PSp_4(q)).2$ & $\Omega_7(q)$ & $q$ odd, $A$ in $\mathcal{C}_1$ or $\mathcal{C}_4$ \\
6 & $\Omega_8^-(q^{1/2})$ & $\Omega_7(q)$ & $q$ square, $A$ in $\mathcal{C}_5$ or $\mathcal{S}$ \\
7 & $\Pa_1$, $\Pa_3$, $\Pa_4$ & $\lefthat((q+1)/d\times\Omega_6^-(q)).2^d$ & $d=(2,q-1)$, $B$ in $\mathcal{C}_1$ or $\mathcal{C}_3$ \\
%%8 & $((q+1)\times\Omega_6^-(q)).2$ & $\Omega^-_8(q^{1/2})$& $q\in\{4,16\}$, $B$ in $\mathcal{C}_3$ or $\mathcal{S}$  \\
%\hline
%8 & $\Omega_7(2)$, $\Pa_1$, $\Pa_3$, $\Pa_4$, &$\A_9$  & $q=2$ \\
% & $(3\times\Omega_6^-(2)).2$ & & \\
%9 & $(\SL_2(4)\times\SL_2(4)).2^2$ & $\Omega_7(2)$ & $q=2$, $A$ in $\mathcal{C}_2$ or $\mathcal{C}_3$ \\
%10 & $\Omega_7(3)$, $\Pa_1$, $\Pa_3$, $\Pa_4$ & $\Omega_8^+(2)$ & $q=3$ \\
%11 & $\Pa_{13}$, $\Pa_{14}$, $\Pa_{34}$ & $\Omega_8^+(2)$ & $q=3$, $G\geqslant L.2$ \\
%12 & $\Pa_1$, $\Pa_3$, $\Pa_4$ & $2^6.\A_8$ & $q=3$, $G\geqslant L.2$ \\
%13 & $(\SL_2(16)\times\SL_2(16)).2^2$ & $\Omega_7(4)$ & $q=4$, $G\geqslant L.2$ \\
%14 & $(3\times\Omega_6^+(4)).2$, $\Omega_8^-(2)$ & $(5\times\Omega_6^-(4)).2$ & $q=4$, $G\geqslant L.2$ \\
8 & $\Omega_8^-(4)$ & $(17\times\Omega_6^-(16)).2$ & $q=16$ \\
\hline
\end{tabular}
\end{table}}

{\begin{table}[htbp]
\caption{Maximal factorizations of $G$ with $L=\POm_{2m}^+(q)$, $m\geqslant5$}\label{TabMaxOmegaPlus1}
\centering
\begin{tabular}{|l|l|l|l|}
\hline
 & $A\cap L$ & $B\cap L$ & Remark\\
\hline
1 & $\Pa_m$, $\Pa_{m-1}$ & $\N_1$ & \\
2 & $\lefthat\GU_m(q).2$ & $\N_1$ & $m$ even \\
3 & $(\PSp_2(q)\otimes\PSp_m(q)).a$ & $\N_1$ & $m$ even, $q>2$, $a=\gcd(2,m/2,q-1)$ \\
4 & $\Pa_m$, $\Pa_{m-1}$ & $\N_2^-$ & \\
5 & $\lefthat\GU_m(q).2$ & $\Pa_1$ & $m$ even \\
6 & $\lefthat\GL_m(q).2$ & $\N_1$ & $G\geqslant L.2$ if $m$ odd \\
7 & $\Omega_m^+(4).2^2$ & $\N_1$ & $m$ even, $q=2$ \\
8 & $\Omega_m^+(16).2^2$ & $\N_1$ & $m$ even, $q=4$, $G\geqslant L.2$ \\
9 & $\lefthat\GL_m(2).2$ & $\N_2^-$ & $q=2$, $G\geqslant L.2$ if $m$ odd \\
10 & $\lefthat\GL_m(4).2$ & $\N_2^-$ & $q=4$, $G\geqslant L.2$, $G\neq\GO_{2m}^+(4)$ \\
11 & $\lefthat\GU_m(4).2$ & $\N_2^+$ & $m$ even, $q=4$, $G=L.2$ \\
12 & $\Omega_9(q).a$ & $\N_1$ & $m=8$, $a\leqslant2$ \\
13 & $\Co_1$ & $\N_1$ & $m=12$, $q=2$ \\
\hline
\end{tabular}
\end{table}}

Inspecting the candidates in Tables~\ref{TabMaxOmegaPlus2} and~\ref{TabMaxOmegaPlus1} we see that
\[
B^{(\infty)}=q^{2m-2}{:}\Omega_{2m-2}^+(q),\,\ \Omega_{2m-1}(q),\,\ \Omega_{2m-2}^+(q)\,\text{ or }\,\Omega_{2m-2}^-(q).
\]
%Moreover, if $B^{(\infty)}=q^{2m-2}{:}\Omega_{2m-2}^+(q)$ or $\Omega_{2m-2}^+(q)$, then $A^{(\infty)}=\lefthat\SU_m(q)$ and $m$ is even.
We handle the cases $B^{(\infty)}\in\{q^{2m-2}{:}\Omega_{2m-2}^+(q),\Omega_{2m-2}^+(q)\}$, $B^{(\infty)}=\Omega_{2m-1}(q)$ and $B^{(\infty)}=\Omega_{2m-2}^-(q)$, respectively, in the following three lemmas.

\begin{lemma}\label{LemOmegaPlus1Row5,11}
$B^{(\infty)}\not=q^{2m-2}{:}\Omega_{2m-2}^+(q)$ or $\Omega_{2m-2}^+(q)$.
\end{lemma}

\begin{proof}
Suppose that $B^{(\infty)}=q^{2m-2}{:}\Omega_{2m-2}^+(q)$ or $\Omega_{2m-2}^+(q)$, that is, $B\cap L=\Pa_1$ or $\N_2^+$.
Then $A^{(\infty)}=\lefthat\SU_m(q)$ with $m$ even, by Tables~\ref{TabMaxOmegaPlus2} and~\ref{TabMaxOmegaPlus1}.
Computation in \magma~\cite{BCP1997} shows that $L\not=\Omega_{12}^+(2)$.
%Thus $(m,q)\neq(6,2)$.
Let
\[
r\in\ppd(q^{2m-2}-1),\, \ s\in\ppd(q^m-1),\, \ t\in\ppd(q^{m-1}-1)\,\text{ and }\,u\in\ppd(q^{2m-4}-1).
\]
Since $A^{(\infty)}=\lefthat\SU_m(q)$, the index $|L|/|A\cap L|$ is divisible by $tu$.
As $|L|/|B^{(\infty)}|$ is divisible by $(q^m-1)(q^{m-1}+1)$, so is $|L|/|B\cap L|$. Hence $|L|/|B\cap L|$ is divisible by $rs$.
Write $q=p^f$ with prime $p$. Let $M=\Rad(A)$. Then $A/M$ is almost simple with socle $\PSU_m(q)$, and $HM/M$ is a factor of $A/M$ with order divisible by $rs$. By~\cite[Theorem~A]{LPS1990}, there is no core-free factor of $A/M$ with order divisible by $rs$. Hence $HM/M\trianglerighteq\Soc(A/M)=\PSU_m(q)$, and so $H^{(\infty)}=A^{(\infty)}=\lefthat\SU_m(q)$.

Let $N=\Rad(B)$. Then $B/N$ is almost simple with socle $\POm_{2m-2}^+(q)$. Since $|L|/|A\cap L|$ is divisible by $tu$, the group $KN/N$ is a factor of $B/N$ with order divisible by $tu$. Inspecting Table~~\ref{TabMaxOmegaPlus1} for $m\geqslant6$ and Table~\ref{TabMaxLinear} for $m=4$ (note $\POm_6^+(q)\cong\PSL_4(q)$), we see that no almost simple group with socle $\POm_{2m-2}^+(q)$ has a core-free factor of order divisible by $tu$. Therefore, $KN/N\trianglerighteq\Soc(B/N)=\POm_{2m-2}^+(q)$. However, this implies that $|H^{(\infty)}||K^{(\infty}|$ is divisible by $|\PSU_m(q)||\POm_{2m-2}^+(q)|$, contradicting Lemma~\ref{LemXia4}.
\end{proof}

\begin{lemma}\label{LemOmegaPlus1Row1--3,6--8,12--13}
$B^{(\infty)}\not=\Omega_{2m-1}(q)$.
\end{lemma}

\begin{proof}
Suppose that $B^{(\infty)}=\Omega_{2m-1}(q)$.
Suppose further that $K^{(\infty)}=B^{(\infty)}$.
Since $|H^{\infty}|_{\pi'}$ divides $|L|_{\pi'}/|K^{(\infty)}|_{\pi'}$ by Lemma~\ref{LemXia4}, inspecting the candidates of $H$ in~\cite[Lemma~4.5]{LPS2010}, we have that either $H\cap L=\lefthat(P.\SL_2(q^{m/2}))$, or $H\cap L=\lefthat(P.\G_2(q^{m/6})')$ with $q$ even, where $P\leqslant q^{m(m-1)/2}$. In particular, $m$ is even. It follows that $K^{(\infty)}=\Omega_{2m-1}(q)$ is maximal in $L$, and so $K$ is maximal in $G$ as $G=KL$. If $q$ is even, then $H\cap K=1$, which contradicts~\cite[Theorem~1.1]{LPS2010}.
Thus $q$ is odd, and $H\cap L=\lefthat(P.\SL_2(q^{m/2}))$. Note that $p\notin\pi=\{2\}$. We have $|H|_p|K|_p=|G|_p$. If $P=1$, then $|H\cap L|_p=q^{m/2}$, which implies that $|H|_p<q^{m-1}=|G|_p/|K|_p$, a contradiction. Hence $P\neq1$, and so $|P|\geqslant q^m$. However, this implies that $|H|_p>|G|_p/|K|_p$, again a contradiction.

Thus $K^{(\infty)}<B^{(\infty)}=\Omega_{2m-1}(q)$. Let $N=\Rad(B)$, and let $\overline{\phantom{x}}\colon B\to B/N$ be the quotient modulo $N$. Then $\overline{B}$ is an almost simple group with socle $\Omega_{2m-1}(q)$ (note that $\Omega_{2m-1}(q)\cong\Sp_{2m-2}(q)$ if $q$ is even), $\overline{K}$ is a core-free factor of $\overline{B}$, and $\overline{B}=\overline{(H\cap B)}\,\overline{K}$. If $q$ is odd, then this is a near-exact factorization of $\overline{B}$, contradicting our inductive hypothesis. Thus $q$ is even. As $B^{(\infty)}=\Omega_{2m-1}(q)$, we see from Tables~\ref{TabMaxOmegaPlus2} and~\ref{TabMaxOmegaPlus1} that $A^{(\infty)}$ is one of the following:
\begin{align}
&q^{m(m-1)/2}{:}\SL_m(q),\,\ \SL_m(q),\,\ \SU_m(q)\text{ with }m\text{ even},\,\ \Omega_7(q)\text{ with }m=4,\label{EqnXia14}\\
&\PSp_2(q)\times\PSp_m(q),\,\ \Omega_9(q)\text{ with }m=8,\,\ \Omega_8^-(q^{1/2})\text{ with }m=4,\label{EqnXia15}\\
&\Omega_m^+(q^2)\text{ with }q\in\{2,4\},\,\ \Co_1\text{ with }(m,q)=(12,2).\label{EqnXia16}
\end{align}
Let $r\in\ppd(q^{2m-2}-1)$, $s\in\ppd(q^m-1)$, $t\in\ppd(q^{m-1}-1)$ and $u\in\ppd(q^{2m-4}-1)$.

Suppose that $A^{(\infty)}$ is as in~\eqref{EqnXia15} or~\eqref{EqnXia16}. If $A^{(\infty)}$ lies in~\eqref{EqnXia15}, then $|G|/|A|$ is divisible by $rtu$, and so is $|\overline{K}|$. If $A^{(\infty)}$ lies in~\eqref{EqnXia16}, then $|G|/|A|$ is divisible by $rt$ and $\ppd(q^{2m-6}-1)$, and so is $|\overline{K}|$. In either case, however, we see from Table~\ref{TabMaxSymplectic} that there is no such core-free factor $\overline{K}$ of $\overline{B}$, a contradiction.

Thus we conclude that $A^{(\infty)}$ is in~\eqref{EqnXia14}. Let $\overline{B}_1$ be a maximal core-free subgroup of $\overline{B}$ containing $\overline{K}$, and let $B_1$ be a subgroup of $B$ containing $N$ such that $B_1/N=\overline{B}_1$. Then $\overline{B}=(\overline{A\cap B})\overline{B}_1$, and so $\overline{B}_1$ is a factor of $\overline{B}$. Moreover, as $K\leqslant B_1$, we have $B_1=(A\cap B_1)K$ and hence $\overline{B}_1=(\overline{A\cap B_1})\overline{K}$. For $L=\Omega_{10}^+(2)$ or $\Omega_{12}^+(2)$ computation in \magma~\cite{BCP1997} shows that no such near-exact factorization $G=HK$ exists. Thus $(m,q)\notin\{(5,2),(6,2)\}$. Let $M=\Rad(A)$. We proceed by different candidates for $A^{(\infty)}$ in~\eqref{EqnXia14}.

\textsc{Case}~1: $A^{(\infty)}=q^{m(m-1)/2}{:}\SL_m(q)$ or $\SL_m(q)$. Then $A/M$ is an almost simple group with socle $\PSL_m(q)$, and the index $|G|/|A|$ is divisible by $r\in\ppd(q^{2m-2}-1)$ and $u\in\ppd(q^{2m-4}-1)$. Hence $|\overline{K}|$ is divisible by $ru$, and so is $|\overline{B}_1|$. Then by Table~\ref{TabMaxSymplectic}, the maximal core-free factor $\overline{B}_1$ of $\overline{B}$ satisfies $\overline{B}_1^{(\infty)}=\Omega_{2m-2}^-(q)$. Since $\overline{B}_1=(\overline{A\cap B_1})\overline{K}$ with $|\overline{K}|$ divisible by $ru$, it follows from~\cite[Theorem~A]{LPS1990} that $K^{(\infty)}=\overline{B}_1^{(\infty)}=\Omega_{2m-2}^-(q)$.
Hence $K\leqslant\Nor_G(K^{(\infty)})=\N_2^-[G]$, up to a triality automorphism of $L$ if $m=4$. Thus $|G|/|K|$ is divisible by $s\in\ppd(q^m-1)$ and $t\in\ppd(q^{m-1}-1)$, and so $HM/M$ is a factor of $A/M$ with order divisible by $st$.
So we conclude from Table~\ref{TabMaxLinear} that $HM/M\trianglerighteq\Soc(A/M)=\PSL_m(q)$.
This yields that $|H^{(\infty)}||K^{(\infty)}|$ is divisible by $|\PSL_m(q)||\Omega_{2m-2}^-(q)|$, contradicting Lemma~\ref{LemXia4}.

\textsc{Case}~2: $A^{(\infty)}=\SU_m(q)$ with $m$ even. In this case, $A/M$ is an almost simple group with socle $\PSU_m(q)$, and $|L|/|A\cap L|$ is divisible by $t\in\ppd(q^{m-1}-1)$ and $u\in\ppd(q^{2m-4}-1)$. This implies that $|K|$ is divisible by $tu$, and so is $|\overline{K}|$. Since $\overline{K}\leqslant\overline{B}_1$, the maximal core-free subgroup $\overline{B}_1$ of the almost simple symplectic group $\overline{B}$ is a factor of $\overline{B}$ with order divisible by $tu$. Then we conclude by Table~\ref{TabMaxSymplectic} that $\overline{B}_1$ is an almost simple group with socle $\Omega_{2m-2}^+(q)$. Since $\overline{B}_1=(\overline{A\cap B_1})\overline{K}$ with $|\overline{K}|$ divisible by $tu$, we then conclude by Table~\ref{TabMaxOmegaPlus1} that $\overline{K}\trianglerighteq\Soc(\overline{B}_1)=\Omega_{2m-2}^+(q)$. Therefore, $K^{(\infty)}=\Omega_{2m-2}^+(q)$.
As a consequence, $|G|/|K|$ is divisible by $r\in\ppd(q^{2m-2}-1)$ and $s\in\ppd(q^m-1)$.
Hence $|H|$ is divisible by $rs$, and so is $|HM/M|$. Since $HM/M$ is a factor of $A/M$ with $\Soc(A/M)=\PSU_m(q)$, we conclude by~\cite[Theorem~A]{LPS1990} that $HM/M\trianglerighteq\Soc(A/M)=\PSU_m(q)$, which implies $H^{(\infty)}=\SU_m(q)$.
This implies that $|H^{(\infty)}||K^{(\infty)}|$ is divisible by $|\SU_m(q)||\Omega_{2m-2}^+(q)|$, contradicting Lemma~\ref{LemXia4}.

\textsc{Case}~3: $A^{(\infty)}=\Omega_7(q)$ with $m=4$. In this case, $A=(\N_1[G])^\alpha$ and $B=(\N_1[G])^\beta$ for some triality automorphisms $\alpha$ and $\beta$ of $L$. As $K^{(\infty)}<B^{(\infty)}$, similar argument shows that $H^{(\infty)}<A^{(\infty)}$. Since $|G|/|A|=|G|/|B|$ is divisible by $s\in\ppd(q^4-1)$, so are $|H|$ and $|K|$. Hence $HM/M$ and $KN/N$ are core-free factors, respectively, of $A/M$ and $B/N$ with order divisible by $s$. Let $A_1$ be a subgroup of $A$ containing $M$ such that $A_1/M$ is a maximal core-free subgroup of $A/M$ containing $HM/M$, and recall the definition of $B_1$ before Case~1. Then $H\leqslant A_1<A$ and $K\leqslant B_1<B$. This implies that $|A_1|$ and $|B_1|$ are both divisible by $s$, and so by Table~\ref{TabMaxSymplectic} we have
\[
A_1^{(\infty)},B_1^{(\infty)}\in\{q^5{:}\Omega_5(q),\Omega_6^+(q),\Omega_6^-(q),\Omega_5(q)\}.
\]
Note that $|H||K|$ is divisible by $r\in\ppd(q^6-1)$ and $t\in\ppd(q^3-1)$ as $|G|$ is. We derive that $|A_1||B_1|$ is divisible by $rt$, and so $A_1^{(\infty)}=\Omega_6^\varepsilon(q)$ and $B_1^{(\infty)}=\Omega_6^{-\varepsilon}(q)$, where $\varepsilon\in\{+,-\}$. Since $A=(\N_1[G])^\alpha$ and $B=(\N_1[G])^\beta$, it follows that $A_1\leqslant\N_1^\varepsilon[A]\leqslant(\N_2^\varepsilon[G])^\alpha$ and $B_1\leqslant\N_1^{-\varepsilon}[B]\leqslant(\N_2^{-\varepsilon}[G])^\beta$. Thus $G=HK=A_1B_1=(\N_2^\varepsilon[G])^\alpha(\N_2^{-\varepsilon}[G])^\beta$. However, the factorization $G=(\N_2^\varepsilon[G])^\alpha(\N_2^{-\varepsilon}[G])^\beta$ does not exist by Table~\ref{TabMaxOmegaPlus2}, a contradiction.
\end{proof}

\begin{lemma}\label{LemOmegaPlus1Row4,9,10}
$B^{(\infty)}\not=\Omega_{2m-2}^-(q)$.
\end{lemma}

\begin{proof}
Suppose $B^{(\infty)}=\Omega_{2m-2}^-(q)$.
Then $A^{(\infty)}$ is one of the groups by Tables~\ref{TabMaxOmegaPlus2} and~\ref{TabMaxOmegaPlus1}:
\[
\lefthat(q^{m(m-1)/2}{:}\SL_m(q)),\,\ \SL_m(q)\text{ with }q\in\{2,4\},\,\ \Omega_8^-(4)\text{ with }(m,q)=(4,16).
\]
In particular, $|G|/|A|$ is divisible by $r\in\ppd(q^{2m-2}-1)$ and $u\in\ppd(q^{2m-4}-1)$, which implies that $|K|$ is divisible by $ru$. Let $N=\Rad(B)$. Then $B/N$ is almost simple with socle $\POm_{2m-2}^-(q)$, and $KN/N$ is a factor of $B/N$ with order divisible by $ru$. By~\cite[Theorem~A]{LPS1990}, no almost simple group with socle $\POm_{2m-2}^-(q)$ has a core-free factor of order divisible by $ru$. Thus $KN/N\trianglerighteq\Soc(B/N)=\POm_{2m-2}^-(q)$, and so $K^{(\infty)}=B^{(\infty)}=\Omega^-_{2m-2}(q)$.

If $A^{(\infty)}=\SL_m(q)$ with $q\in\{2,4\}$ or $A^{(\infty)}=\lefthat(q^{m(m-1)/2}{:}\SL_m(q))$, then similar argument as in Case~1 in the proof of Lemma~\ref{LemOmegaPlus1Row1--3,6--8,12--13} shows that this is not possible. Now assume that $A^{(\infty)}=\Omega_8^-(4)$ with $(m,q)=(4,16)$. Since $|G|/|B|$ is divisible by $s\in\ppd(16^4-1)$ and $t\in\ppd(16^3-1)$, so is $|H|$. Let $M=\Rad(A)$. Then $A/M$ is almost simple with socle $\Omega_8^-(4)$, and $HM/M$ is a factor of $A/M$ with order divisible by $st$. According to Table~\ref{TabMaxOmegaMinus}, there is no almost simple group with socle $\Omega_8^-(4)$ that has a core-free factor of order divisible by $st$. Hence $HM/M\trianglerighteq\Soc(A/M)=\Omega_8^-(4)$. This implies that $H^{(\infty)}=A^{(\infty)}=\Omega_8^-(4)$, which yields that $|H^{(\infty)}||K^{(\infty)}|$ is divisible by $|\Omega_8^-(4)||\Omega_6^-(16)|$, contradicting Lemma~\ref{LemXia4}.
\end{proof}

\subsection{Symplectic groups}
\ \vspace{1mm}

Let $L=\Soc(G)=\PSp_{2m}(q)$, where $m\geqslant2$ and $(m,q)\neq(2,2)$, and let $\pi=\pi(\sfM(L))$.
Then Schur multiplier $\mathsf{M}(L)$ is either $\ZZ_{(2,q-1)}$ or $\ZZ_2$ for $L=\Sp_6(2)$.
In particular, $\sfM(L)$ is a 2-group.

\begin{proposition}\label{ThmSymplectic}
Let $G=HK$ be a near-exact factorization with nonsolvable $H$ and $K$.
Then, interchanging $H$ and $K$ if necessary, $(G,H,K)$ lies in the following table, where $\calO_1\calO_2=\ZZ_2$.
\[
\begin{array}{|llll|}
\hline
G & H & K & H\cap K \\
\hline
\Sp_4(4).2 & \Sy_5 & \SL_2(16).4 & 1\\ \hline
\Sp_6(2) & \Sy_5\times\calO_1 & \SU_3(3).\calO_2 & |\calO_1||\calO_2|/2 \\
\Sp_6(2) & 2^4.\A_5\leqslant H\leqslant 2^5.\Sy_5 & \SU_3(3)\leqslant K\leqslant\SU_3(3).2 & 4\leqslant H\cap K\leqslant8{:}2^2 \\
\Sp_6(2) & 2^4.\A_5\leqslant H\leqslant 2^5.\Sy_5 & \SL_2(8).3 & H\cap K\leqslant2^2 \\
\Sp_6(2) & 2^5.\A_6\leqslant H\leqslant 2^5.\Sy_6 & \SL_2(8) & 2^2\leqslant H\cap K\leqslant2^3 \\
\hline
\Sp_6(4).2 & \SL_2(16).4 & \G_2(4).2 & 1\\ \hline
\Sp_8(2) & \Sy_5 & \GO_8^-(2) & 1 \\
\hline
\end{array}
\]
\end{proposition}

Computation in \magma~\cite{BCP1997} shows that $(G,H,K)$ satisfies Proposition~$\ref{ThmSymplectic}$ for
\[\mbox{$L=\Sp_4(3)$,\, $\Sp_4(4)$,\, $\Sp_6(2)$,\, $\PSp_6(3)$,\, $\Sp_6(4)$\, or\, $\Sp_8(2)$.}\]
To prove Proposition~\ref{ThmSymplectic}, by Lemmas~\ref{LemXia5} and~\ref{LemXia27}, we may assume that $G=HL=KL$ and there exist maximal subgroups $A$ and $B$ of $G$ containing $H$ and $K$ respectively.
We thus assume in the following that $L$ is not any of these groups.
The maximal factorizations $G=AB$ are classified in \cite{LPS1990}, as recorded in Table~\ref{TabMaxSymplectic} below.

{\small\begin{table}[htbp]
\caption{Maximal factorizations of $G$}\label{TabMaxSymplectic}
\centering
\begin{tabular}{|l|l|l|l|}
\hline
 & $A\cap L$ & $B\cap L$ & Remark\\
\hline
1 & $\PSp_{2a}(q^b).b$ & $\Pa_1$ & $ab=m$, $b$ prime \\
2 & $\Sp_{2a}(q^b).b$ & $\GO_{2m}^+(q)$ & $q$ even, $ab=m$, $b$ prime \\
3 & $\Sp_{2a}(q^b).b$ & $\GO_{2m}^-(q)$ & $q$ even, $ab=m$, $b$ prime \\
4 & $\GO_{2m}^-(q)$ & $\Pa_m$ & $q$ even \\
5 & $\GO_{2m}^-(q)$ & $\Sp_m(q)\wr\Sy_2$ & $m$ even, $q$ even \\
6 & $\GO_{2m}^-(2)$ & $\GO_{2m}^+(2)$ & $q=2$ \\
7 & $\GO_{2m}^-(4)$ & $\GO_{2m}^+(4)$ & $q=4$, $G=L.2$, two classes \\
 & & & of factorizations if $m=2$ \\
8 & $\GO_{2m}^-(4)$ & $\Sp_{2m}(2)$ & $q=4$, $G=L.2$ \\
9 & $\GO_{2m}^-(16)$ & $\Sp_{2m}(4)$ & $q=16$, $G=L.4$ \\
10 & $\Sp_m(4).2$ & $\N_2$ & $m\geqslant6$ even, $q=2$ \\
11 & $\Sp_m(16).2$ & $\N_2$ & $m\geqslant4$ even, $q=4$, $G=L.2$ \\
12 & $\Sz(q)$ & $\GO_4^+(q)$ & $m=2$, $q=2^f$, $f\geqslant3$ odd, \\
 & & & two classes of factorizations \\
13 & $\GO_6^+(q)$, $\GO_6^-(q)$, $\Pa_1$, $\N_2$ & $\G_2(q)$ & $m=3$, $q\geqslant8$ even \\
\hline
\end{tabular}
\end{table}}

As for most candidates in Table~\ref{TabMaxSymplectic}, either $A\cap L$ or $B\cap L$ is one of the three groups $\Pa_1$, $\GO_{2m}^+(q)$ and $\GO_{2m}^-(q)$, we first treat these cases in the following three lemmas.

\begin{lemma}\label{Lem10.2}
$B\cap L\not=\Pa_1$.
\end{lemma}

\begin{proof}
Suppose that $B\cap L=\Pa_1$.
As $G$ is transitive on the set of $1$-dimensional subspaces, so is $H$.
By Hering's classification (see~\cite[Lemma~3.1]{LPS2010}) of such groups $H$, either
\begin{itemize}
\item
$H^{(\infty)}=\PSp_{2a}(q^b)$ with $m=ab$ and $b>1$, or
\item
$H^{(\infty)}=\G_2(q^b)$ with $m=3b$ and $q$ even.
\end{itemize}
In either case, ${|L|/|H^{(\infty)}|}$ is divisible by $\ppd(q^{2m-2}-1)$ and $\ppd(q^{m-1}-1)$, and so is $|K^{(\infty)}|$.
%Moreover, $|H^{(\infty)}|_{\pi'}|B^{(\infty)}|_{\pi'}$ does not divide $|L|$, and so $K^{(\infty)}\neq B^{(\infty)}$.
Note that $B^{(\infty)}=[q^{2m-1}]{:}\Sp_{2m-2}(q)$.
Let $\overline{\phantom{x}}\colon B\to B/\Rad(B)$ be the quotient modulo $\Rad(B)$.
Then $\ov B$ is almost simple with socle $\PSp_{2m-2}(q)$, and $\ov B=\ov K\,\overline{(H\cap B)}$ with $|\ov K|$ divisible by $\ppd(q^{2m-2}-1)$ and $\ppd(q^{m-1}-1)$.
If $\overline{K}$ contains $\PSp_{2m-2}(q)$, then as $H^{(\infty)}=\PSp_{2a}(q^b)$ or $\G_2(q^b)$, the intersection $H\cap K$ is not a 2-group, a contradiction. Hence $\overline{K}$ is core-free in $\overline{B}$.
Since $\overline{K}$ is nonsolvable with order divisible by $\ppd(q^{2m-2}-1)$ and $\ppd(q^{m-1}-1)$, an inspection of Table~\ref{TabMaxSymplectic} shows that $m\geqslant3$ and $\ov K\cap\Soc(\ov B)$ is contained in:
\begin{equation}\label{EqnXia99}
\mbox{$\PSp_{2c}(q^d).d$ with $c<cd=m-1$,\, $\GO_{2m-2}^-(q)$,\, or $\G_2(q)$ with $m=4$.}
\end{equation}
To complete the proof, we show that this is not possible in any of the following cases.

\textsc{Case}~1. $m=3$.
Then $H^{(\infty)}=\PSp_2(q^3)$ or $\G_2(q)$.
If $H^{(\infty)}=\G_2(q)$, then $q$ is even and $G=HK$ is an exact factorization with $H$ maximal in $G$, which is not possible by~\cite[Theorem~1.1]{LPS2010}.
Thus $H^{(\infty)}=\PSp_2(q^3)$, and so $|G|/|H|$ is divisible by $|\PSp_6(q)|/|\PSp_2(q^3)|=q^6(q^4-1)(q^2-1)$.
However, $|K|$ is not divisible by $q^6(q^4-1)(q^2-1)$ by \eqref{EqnXia99}, a contradiction.

\textsc{Case}~2. $m=4$.
In this case, $H^{(\infty)}=\PSp_4(q^2)$ or $\PSp_2(q^4)$.
If $H^{(\infty)}=\PSp_2(q^4)$, then $|G|/|H|$ is divisible by $|\PSp_8(q)|/|\PSp_2(q^4)|=q^{12}(q^6-1)(q^4-1)(q^2-1)$, and so is $|K|$, which is not possible by~\eqref{EqnXia99}.
Therefore, $H^{(\infty)}=\PSp_4(q^2)$. It follows that $|G|/|H|$ is divisible by $|\PSp_8(q)|/|\PSp_4(q^2)|=q^8(q^6-1)(q^2-1)$, and so is $|K|$.
This together with~\eqref{EqnXia99} implies that $K^{(\infty)}=Q.\Omega_6^-(q)$ or $Q.\G_2(q)$, where $q$ is even and $Q$ is a $2$-group with $|Q|>q$.
Since $|\Z(\Rad(B))|=q$, it follows that $|Q|\geqslant q^6$. However, this implies that $|H^{(\infty)}|_2|K^{(\infty)}|_2\geqslant q^{20}>|L|_2$, contradicting Lemma~\ref{LemXia4}.

\textsc{Case}~3. $m\geqslant5$ is odd.
Then $b$ is odd, and so $b\geqslant3$, which implies that $|G|/|H|$ is divisible by $r\in\ppd(q^{2m-2}-1)$ and $s\in\ppd(q^{2m-4}-1)$.
It follows that $|K|$ is divisible by $rs$, and so is $|\overline{K}|$.
Then we conclude from~\eqref{EqnXia99} that $\overline{K}\cap\Soc(\overline{B})\leqslant\GO_{2m-2}^-(q)$.
As a consequence, $|G|/|K|$ is divisible by $\ppd(q^{m-1}-1)$. However, $|H|$ is coprime to $\ppd(q^{m-1}-1)$ as $b$ is an odd prime dividing $m$.
This contradicts $G=HK$.

\textsc{Case}~4. $m\geqslant6$ is even.
If $b$ is odd, then $|L|/|H\cap L|$ is divisible by $\ppd(q^{2m-2}-1)$, $\ppd(q^{m-1}-1)$ and $\ppd(q^{2m-4}-1)$, and so is $|\overline{K}|$. If $b=2$, then $|L|/|H\cap L|$ is divisible by $\ppd(q^{2m-2}-1)$, $\ppd(q^{m-1}-1)$ and $\ppd(q^{2m-6}-1)$, and so is $|\overline{K}|$.
However,~\eqref{EqnXia99} shows that neither is possible.
\end{proof}

\begin{lemma}\label{Lem10.3}
$B\cap L\not=\GO_{2m}^-(q)$.
\end{lemma}

\begin{proof}
Suppose that $B\cap L=\GO_{2m}^-(q)$.
Then $q$ is even, $B$ is almost simple with socle $\Omega_{2m}^-(q)$, and $B=(H\cap B)K$.
By our inductive hypothesis and~\cite[Theorem~3]{BL2021} we derive that $K^{(\infty)}=B^{(\infty)}=\Omega_{2m}^-(q)$.
Since $q$ is even, we have $\mathsf{M}(L)=1$ and so $H\cap K=1$.
Hence Lemma~\ref{LemXia4} implies that $|H^{(\infty)}||K^{(\infty)}|$ divides $|L|$, that is, $|H^{(\infty)}|$ divides $q^m(q^m-1)$.
Then by~\cite[Lemma~4.6]{LPS2010}\footnote{In part~(iv) of Lemma~4.6 in~\cite{LPS2010}, the case $B\trianglerighteq\Omega_{2m}^+(q)$ with $q\in\{2,4\}$ is missing.} we conclude that either
\begin{itemize}
\item $q^m$ is a square and $H^{(\infty)}=\Sp_2(q^{m/2})$, or
\item $m$ is even and $H^{(\infty)}=P.\SL_2(q^{m/2})$ with $P\leqslant q^{m(m+1)/2}$.
\end{itemize}
In particular, $H^{(\infty)}=P.\SL_2(q^{m/2})$ with $P\leqslant q^{m(m+1)/2}$, and $K\cap L=\Omega_{2m}^-(q)$ or $\GO_{2m}^-(q)$. Let $B$ be the maximal subgroup of $G$ containing $K$. Then $B\cap L=\GO_{2m}^-(q)$. If $K\cap L=\GO_{2m}^-(q)$, then $K\cap L$ is maximal in $L$ and so $K$ is maximal in $G$, contradicting~\cite[Theorem~1.1]{LPS2010}. Thus $K\cap L=\Omega_{2m}^-(q)$, and so
\begin{equation}\label{EqnXia11}
|H|=\frac{|G|}{|K|}=\frac{|L|}{|K\cap L|}=\frac{|\Sp_{2m}(q)|}{|\Omega_{2m}^-(q)|}=q^m(q^m-1).
\end{equation}
Consequently, $|P|\leqslant|H|/|\SL_2(q^{m/2})|=q^{m/2}<q^m$. This implies that $\SL_2(q^{m/2})$ acts trivially on $P$. Hence $H^{(\infty)}=P.\SL_2(q^{m/2})=P\times\SL_2(q^{m/2})$. Since $H^{(\infty)}$ is a perfect group, we then obtain $P=1$ and $H^{(\infty)}=\SL_2(q^{m/2})$. Let
\[
S=H^{(\infty)}=\SL_2(q^{m/2}).
\]
Now since $|H/H^{(\infty)}|=q^m(q^m-1)/|\SL_2(q^{m/2})|=q^{m/2}$, we see that $H$ has a nontrivial normal $2$-subgroup $E=\Cen_H(S)$. Then by Borel-Tits Theorem, $H\leqslant\Pa_i[G]$ for some $1\leqslant i\leqslant m$. As $H\geqslant\SL_2(q^{m/2})$, we have $i=m$ with $m$ even. Let $T=\Cen_G(S)$. Then $E\leqslant T$.

For an element $t\in S$ of order $3$, note that $t$ fixes a point of $[G:B]$ if and only if $t$ fixes a point of $[G:K]$, as $|B/K|=2$. Then similar argument as in the proof of~\cite[Proposition~9.1]{LPS2010} shows that $\Cen_L(S)\cong\SL_2(q^{m/2})$, and there exist involutions
\[
u\in\Cen_L(S)\quad\text{and}\quad s\in S
\]
such that $us\in B^x$ for some $x\in G$ and $us$ has Jordan canonical form $J_2^m$, where $J_2$ is a unipotent Jordan block of length $2$. Since $m$ is even, for any element $r\in L$ with Jordan canonical form $J_2^m$, we have $r\in\Omega_{2m}^-(q)$ if and only if $r\in\GO_{2m}^-(q)$ (see~\cite[Lemma~2.2]{LPS2010}). Therefore, $us\in K^x$.
Since no nontrivial field automorphism of $L$ centralizes $S$, we obtain
\[
T=\Cen_G(S)=\Cen_L(S)\cong\SL_2(q^{m/2}).
\]
Then since $\SL_2(q^{m/2})$ has only one conjugacy class of involutions, there exists $y\in T$ such that $u^y\in E$. Now $u^ys\in ES\leqslant H$, and as $y$ centralizes $S$, we have
\[
u^ys=u^ys^y=(us)^y\in(K^x)^y=K^{xy}.
\]
It follows that $u^ys$ is an involution in $H\cap K^{xy}$. However, the exact factorization $G=HK$ implies that $H\cap K^{xy}=1$, a contradiction.
\end{proof}

\begin{lemma}\label{Lem10.4}
$B\cap L\not=\GO_{2m}^+(q)$.
\end{lemma}

\begin{proof}
Suppose that $B\cap L=\GO_{2m}^+(q)$.
Since $q$ is even, we have $\mathsf{M}(L)=1$ and so $H\cap K=1$.
Hence Lemma~\ref{LemXia4} implies that $|H^{(\infty)}||K^{(\infty)}|$ divides $|L|$.
By Lemma~\ref{Lem10.3} we may assume that $A\cap L\neq\GO_{2m}^-(q)$.
Then Table~\ref{TabMaxSymplectic} shows that $A$ is almost simple with one of the following socles:
\begin{enumerate}[{\rm (i)}]
\item
$\Sp_{2a}(q^b)$ with $m=ab$ and $b$ prime, as in rows~2 of Table~\ref{TabMaxSymplectic};
%\item $\GO_{2m}^-(q)$ with $q\in\{2,4\}$, from rows~6-7 of Table~\ref{TabMaxSymplectic};
\item
$\Sz(q)$ with $m=2$ and $q=2^f\geqslant2^3$, as in row~~12 of Table~\ref{TabMaxSymplectic};
\item
$\G_2(q)$ with $m=3$ and $q$ even, as in row~~13 of Table~\ref{TabMaxSymplectic}.
\end{enumerate}

{\sc Case 1.}\ Assume that $m=2$. In this case, $B\cap L=(\SL_2(q)\times\SL_2(q)).2$, and $A^{(\infty)}=\Sp_2(q^2)$ or $\Sz(q)$.
It follows that $|G|/|B|=q^2(q^2+1)/2$, and so $|H|$ is divisible by $\ppd(q^4-1)$.
By \cite[Table~8.1]{BHR2013} and \cite[Table~8.16]{BHR2013}, we have $H^{(\infty)}=A^{(\infty)}=\Sp_2(q^2)$ or $\Sz(q)$.
If $H^{(\infty)}=A^{(\infty)}=\Sp_2(q^2)$, then applying the graph automorphism of $L$ if necessary, we have that $A\cap L=\GO_4^-(q)$, contradicting Lemma~\ref{Lem10.3}.
Thus $H^{(\infty)}=\Sz(q)$, which implies that $H\cap L=\Sz(q)$ and so $|K|=|G|/|H|=|L|/|H\cap L|=q^2(q^2-1)(q+1)$.
However, as $K$ is a nonsolvable subgroup of $B$, we see from \cite[Table~8.1]{BHR2013} that this is not possible.

{\sc Case~2.}\,  Assume that $m\geqslant3$.
In this case, $B$ is an almost simple group with socle $\Omega_{2m}^+(q)$, and $B=(H\cap B)K$.
By our inductive hypothesis and~\cite[Theorem~3]{BL2021} we derive that either $K^{(\infty)}=B^{(\infty)}=\Omega_{2m}^+(q)$, or $(m,q)=(4,4)$ and $K^{(\infty)}=\Sp_6(4)$.
For the latter, $A$ is an almost simple group with socle $\Sp_4(16)$, and $\Sp_6(4)\leqslant K\cap L\leqslant\Sp_6(4).2$.
It follows that $A$ has no nonsolvable subgroup $H$ such that $|H||K|=|G|$, a contradiction.

We thus have $K^{(\infty)}=B^{(\infty)}$, and hence $|H^{(\infty)}|$ divides $|L|/|K^{(\infty)}|=q^m(q^m+1)$.
As a consequence, $H^{(\infty)}<A^{(\infty)}$. Then since $A=H(A\cap K)$ with $H\cap(A\cap K)=1$, we deduce from~\cite[Theorem~1]{HLS1987} that $A^{(\infty)}\neq\G_2(q)$ or $\Sz(q)$. Therefore, $A$ is an almost simple group with socle $\Sp_{2a}(q^b)$.
By Hypothesis~\ref{hypo} and~\cite[Theorem~3]{BL2021}, either $(A^{(\infty)},H^{(\infty)},q^m)=(\Sp_6(4),\G_2(4),64)$, or $H^{(\infty)}\cong\Omega_{2a}^-(q^b)$. However, neither candidates for $|H^{(\infty)}|$ divides $q^m(q^m+1)$, a contradiction.
\end{proof}

If neither $A\cap L$ nor $B\cap L$ is in $\{\Pa_1,\GO_{2m}^+(q),\GO_{2m}^-(q)\}$, then one of the rows~10,~11,~13 of Table~\ref{TabMaxSymplectic} appears.

\begin{lemma}\label{Lem10.5}
Let $A\cap L=\Sp_m(q^2).2$ and $B\cap L=\N_2$ with $q\in\{2,4\}$, as in rows~$10$--$11$ of Table~$\ref{TabMaxSymplectic}$.
Then there is no such near-exact factorization $G=HK$.
\end{lemma}

\begin{proof}
As $q\in\{2,4\}$, we have $m\geqslant4$.
Since $|L|/|A^{(\infty)}|=|\Sp_{2m}(q)|/|\Sp_m(q^2)|$ is divisible by
\[
r\in\ppd(q^{2m-2}-1),\,\ s\in \ppd(q^{m-1}-1)\ \text{ and }\ t\in\ppd(q^{2m-6}-1),
\]
so is $|K|$. Let $N$ be a normal subgroup of $B$ such that $B/N$ is an almost simple group with socle $\Sp_{2m-2}(q)$, and let $\overline{\phantom{x}}\colon B\to B/N$ be the quotient modulo $N$.
Then $|N|$ is coprime to each of $r$, $s$ and $t$, and so $|\overline{K}|$ is divisible by $rst$.

Suppose that $\overline{K}\ngeqslant\Soc(\overline{B})$. Then $\overline{K}$ is a core-free factor of $\overline{B}$ with order divisible by $rst$. Inspecting the maximal factorizations of the symplectic group $\overline{B}$, listed in Table~\ref{TabMaxSymplectic} with $2m$ replaced by $2m-2$, we see that $\overline{B}$ has no core-free factor of order divisible by $rst$, which is a contradiction.

Thus $\overline{K}\trianglerighteq\Soc(\overline{K})=\Sp_{2m-2}(q)$.
Since $G=HB=H\N_2[G]$, we see from \cite[Lemma~4.2]{LPS2010} that one of the following occurs:
\begin{itemize}
\item $H^{(\infty)}=\Sp_m(q^2)$, or $\Sp_{m/2}(q^4)$ with $q=2$;
\item $H^{(\infty)}=\G_2(q^2)$ with $m=6$, or $\G_2(q^4)$ with $(m,q)=(12,2)$.
\end{itemize}
Then $|H^{(\infty)}||K^{(\infty)}|$ does not divide $|L|$, which is a contradiction by Lemma~\ref{LemXia4}.
\end{proof}

\begin{lemma}\label{Lem10.6}
Let $B\cap L=\G_2(q)$ with $q\geqslant8$ even, as in row~~$13$ of Table~$\ref{TabMaxSymplectic}$.
Then there is no such near-exact factorization $G=HK$.
\end{lemma}

\begin{proof}
Since $q$ is even, we have $\mathsf{M}(L)=1$ and so $H\cap K=1$. It follows that $B=(H\cap B)K$ is an exact factorization.
Then we derive from~\cite{HLS1987} that $K\cap L=B\cap L=\G_2(q)$. Hence $K$ is maximal in $G$. However, by~\cite{LPS2010}, there is no such exact
factorization $G=HK$.
\end{proof}

\section{Exact factorizations of almost simple groups}\label{SecXia6}

We produce a classification of exact factorizations of almost simple groups in this section.
%, based on the results obtained in Propositions~\ref{ThmLinear},~\ref{ThmUnitary},~\ref{ThmOmega},~\ref{ThmOmegaPlus}, and \ref{ThmSymplectic}.

%
%Based on the results obtained in Section~\ref{SecXia7}, we are able to completely classify exact factorizations of almost simple groups.

\subsection{Example}
\ \vspace{1mm}

In this subsection we present the examples of exact factorizations $G=HK$ in rows~2, 12, 13, 15 and 19--21 in Table~\ref{TabXia2}.
In what follows, $\calO$ will be a subgroup of $\mathrm{C}_2$, and $\calO_1$ and $\calO_2$ are subgroups of $\calO$ with
\[
|\calO|=|\calO_1||\calO_2|.
\]

\begin{example}\label{ExaRow2}
Let $q\geqslant7$ be a prime power, let $G=\A_{q+1}$ or $\Sy_{q+1}$, let $\calO=G/\A_{q+1}\leqslant\mathrm{C}_2$, and let $H$ be a subgroup of $G$ such that either $H=\PSL_2(q).(2,q-1)$ is sharply $3$-transitive or $H=\PSL_2(q).((2,q-1)\times2)\nleqslant\A_{q+1}$ with $H\cap\A_{q+1}$ sharply $3$-transitive. In the former case let $K=\A_{q-2}.\calO$ be a subgroup of $G$ that intersects trivially with $H$; in the latter case let $K=\A_{q-2}$.
%If $\calO_2=1$, then let $H=\PSL_2(q).((2,q-1)\times\calO_1)$ be a subgroup of $G$ such that $H\cap\A_{q+1}$ is a $3$-transitive group of index $|\calO_1|$ in $H$ (thus $H\cap\A_{q+1}$ is sharply $3$-transitive as $|H\cap\A_{q+1}|=(q+1)q(q-1)$), and let $K=\A_{q-2}<G$. If $\calO_2=2$, then let $H=\PSL_2(q).(2,q-1)$ be a $3$-transitive subgroup of $G=\Sy_{q+1}$ (thus $H$ is sharply $3$-transitive as $|H|=(q+1)q(q-1)$), and let $K$ be a group in one of the two $G$-conjugacy classes of subgroups $\A_{q-2}.2$ that intersects trivially with $H$.
\end{example}

The next five examples of exact factorizations $G=HK$ are verified by computation in \magma~\cite{BCP1997}.

\begin{example}\label{ExaRow12}
Let $L=\SL_4(4)$, let $\phi$ and $\gamma$ be the field automorphism and graph automorphism of $L$ respectively, let $G$ be an almost simple group with socle $L$ such that
\[
G/L\in\{\langle\phi\rangle,\langle\phi\gamma\rangle,\langle\phi\rangle\times\langle\gamma\rangle\},
\]
and let $\calO=\mathrm{C}_{|G/L|/2}\leqslant\mathrm{C}_2$. Let $H$ be a field extension subgroup $(\SL_2(16).4)\times\calO_1$, and let $K$ be a subgroup $\SL_3(4).(2\times\calO_2)$ that is not in $L.\langle\gamma\rangle$.
\end{example}

\begin{example}\label{ExaRow13}
Let $L=\SU_4(4)$, and let $G=\Aut(L)=L.4$. Let $H$ be a subgroup isomorphic to $\SL_2(16).4$ in a maximal subgroup $(\SL_2(16).4)\times\Sy_3$ of $G$, and let $K=\SU_3(4).4<G$.
\end{example}

\begin{example}\label{ExaRow15}
Let $G=\Sp_6(2)$, and let $\calO\leqslant\mathrm{C}_2$. Let $H$ be a subgroup $\Sy_5\times\calO$ of $G$ that is in both $\GO_6^+(2)$ and $\N_2[G]$, and let $K=\SU_3(3).(2/\calO)\leqslant\G_2(2)=\SU_3(3).2<G$.
\end{example}

\begin{example}\label{ExaRow19}
Let $G=\Omega_8^+(2)$ or $\GO_8^+(2)$, and let $\calO=G/\Omega_8^+(2)\leqslant\mathrm{C}_2$.
\begin{enumerate}[{\rm (a)}]
\item Let $K$ be a maximal core-free subgroup $\A_9.\calO_2$ of $G$, and let $H$ be a subgroup $2^4{:}\A_5.\calO_1$ that intersects trivially with $K$.
\item Let $K$ be a maximal core-free subgroup $\Sp_6(2).\calO_2$ of $G$, and let $H$ be a subgroup $\Sy_5\times\calO_1$ that intersects trivially with $K$.
\end{enumerate}
In both~(a) and~(b), $H$ stabilizes a totally singular $i$-space for some $i\in\{1,3,4\}$.
\end{example}

\begin{example}\label{ExaRow21}
Let $L=\Omega_8^+(4)$, let $\phi$ be the field automorphism of $L$, and let $\gamma$ be an involutory graph automorphism of $L$.
Let $G$ be almost simple with socle $L$ such that
\[
G/L\in\{\langle\phi\rangle,\langle\phi\rangle\times\langle\gamma\rangle\},
\]
and let $\calO=\mathrm{C}_{|G/L|/2}\leqslant\mathrm{C}_2$. Let $K$ be a maximal core-free subgroup $(\Sp_6(4).2).\calO_2$ of $G$, and let $H$ be a subgroup $(\SL_2(16).4)\times\calO_1$ that intersects trivially with $K$. Such a group $H$ stabilizes a totally singular $i$-space for some $i\in\{1,3,4\}$.
\end{example}

\subsection{Proof Theorem~\ref{ThmExact}}
\ \vspace{1mm}

For classical groups of Lie type, as exact factorizations are necessarily near-exact factorizations, the candidates of exact factorizations are provided by Theorem~\ref{ThmNearExact}.

Exact factorizations of alternating groups and symmetric groups are classified in \cite{WW1980}. We combine this result with~\cite[Tables~7.3--7.4]{Cameron1999} and~\cite[Theorem~D]{LPS1990} to formulate an explicit list of $(G,H,K)$ for alternating groups $L$.

\begin{lemma}\label{Sym-exact}
Let $G$ be an almost simple group with socle $\A_n$, and let $H$ and $K$ be nonsolvable subgroups of $G$. If $G=HK$ is an exact factorization, then $(G,H,K)$ lies in rows~\emph{1--11} of Table~$\ref{TabXia2}$ (with $H$ and $K$ possibly interchanged).
\end{lemma}

To conclude this subsection we give:

\begin{proof}[Proof of Theorem~$\ref{ThmExact}$]
Let $G$ be an almost simple group with socle $L$. If $L$ is a classical group of Lie type, then Theorem~\ref{ThmNearExact} shows that $L$ is one of:
\[
\SL_4(2),\,\SL_4(4),\,\SU_4(4),\,\Omega_8^+(2),\,\Omega_8^+(4),\,\Sp_4(4),\,\Sp_6(2),\,\Sp_6(4),\,\Sp_8(2).
\]
For these groups, computation in \magma~\cite{BCP1997} shows that $G=HK$ is an exact factorization if and only if $(G,H,K)$ lies in rows~12--21 of Table~\ref{TabXia2} (with $H$ and $K$ possibly interchanged).
If $L$ is an alternating group, then Lemma~~\ref{Sym-exact} leads to rows~1--11 of Table~\ref{TabXia2}, and conversely these rows give rise to exact factorizations $G=HK$ by~\cite[Theorems~A~and~S]{WW1980}. For sporadic groups, the exact factorizations are classified in~\cite[Theorem~1.3]{Giudici2006}, which gives $G=\M_{24}$, $H=2^4{:}\A_7$ or $\PSiL_3(4)$, and $K=\PSL_2(23)$, as in rows~22--23 of Table~$\ref{TabXia2}$. Finally, checking the condition $|G|=|H||K|$ for the factorizations $G=HK$ classified in~\cite{HLS1987} shows that $L$ is not an exceptional group of Lie type. This completes the proof.
\end{proof}

\section{Biperfect semisimple Hopf algebras}\label{SecHopf}

Let $\mathbb{F}$ be a field, and let $A$ be a bialgebra over $\mathbb{F}$ with multiplication $m\colon A\otimes A\to A$, unit $u\colon\mathbb{F}\to A$, comultiplication $\Delta\colon A\to A\otimes A$, and counit $\epsilon\colon A\to\mathbb{F}$. If there exists a $\mathbb{F}$-linear map $S\colon A\to A$ such that
\[
m\circ(S\otimes\mathrm{id})\circ\Delta=m\circ(\mathrm{id}\otimes S)\circ\Delta=u\circ\epsilon
\]
(the composition $f\circ g$ means $g$ followed by $f$), then $A$ is called a \emph{Hopf algebra} with \emph{antipode}~$S$.

\subsection{Bicrossproduct Hopf algebra}
\ \vspace{1mm}

Let $G=HK$ be an exact factorization. Then for each $x\in H$ and $y\in K$, there exist a unique element $x\triangleleft y$ in $H$ and a unique element $x\triangleright y$ in $K$ such that
\[
yx=(x\triangleleft y)(x\triangleright y).
\]
This defines a map $\triangleleft\colon H\times K\to H$ and a map $\triangleright\colon H\times K\to K$.
Now let
\[
V=\mathbb{C}[K]^*\otimes\mathbb{C}[H]
\]
be a vector space over $\mathbb{C}$. Define a $\mathbb{C}$-linear map $m\colon V\otimes V\to V$ by letting
\[
m(\varphi\otimes a,\psi\otimes b)=(\varphi(a\cdot\psi))\otimes(ab)\ \text{ for }\,\varphi,\psi\in\mathbb{C}[K]^*\text{ and }a,b\in H,
\]
where $\cdot$ denotes the associated action of $H$ on the algebra $\mathbb{C}[K]^*$ and $\varphi(a\cdot\psi)$ is the multiplication of $\varphi$ and $a\cdot\psi$ in the algebra $\mathbb{C}[K]^*$. More precisely, the associated action of $H$ on $\mathbb{C}[K]^*$ is given by
\[
\mbox{$(a\cdot\psi)(x)=\psi(a^{-1}\cdot x)$ for $a\in H$, $\psi\in\mathbb{C}[K]^*$ and $x\in\mathbb{C}[K]$,}
\]
where the second $\cdot$ denotes the action of $H$ on $\mathbb{C}[K]$, extended by $H$ acting on $K$ by $\rhd$.
Identify the vector spaces
\[
V\otimes V=(\mathbb{C}[K]\otimes\mathbb{C}[K])^*\otimes(\mathbb{C}[H]\otimes\mathbb{C}[H])=\mathrm{Hom}_\mathbb{C}(\mathbb{C}[K]\otimes\mathbb{C}[K],\mathbb{C}[H]\otimes\mathbb{C}[H])
\]
in the canonical way. Then define a $\mathbb{C}$-linear map $\Delta\colon V\to V\otimes V$ by letting
\[
(\Delta(\varphi\otimes a))(b\otimes c)=(\varphi(bc)a)\otimes(a\triangleleft b^{-1})\ \text{ for }\,\varphi\in\mathbb{C}[K]^*,\ a\in H\text{ and }b,c\in K.
\]
With $V$, $m$ and $\Delta$ defined above,
we have the following result (see~\cite[Theorem~2.1]{EGGS2000}):

\begin{theorem}[\cite{Kac1968,Takeuchi1981}]
For each exact factorization $G=HK$, there exists a unique semisimple Hopf algebra structure on $V$ with multiplication $m$ and comultiplication $\Delta$.
\end{theorem}

The semisimple Hopf algebra $V$ with multiplication $m$ and comultiplication $\Delta$ is called the \emph{bicrossproduct Hopf algebra} associated with the exact factorization $G=HK$. A Hopf algebra $A$ is said to be \emph{perfect} if there is no nontrivial one-dimensional $A$-module, and said to be \emph{biperfect} if both $A$ and its dual are perfect. It was shown in~\cite[Theorem~2.3]{EGGS2000} that the bicrossproduct Hopf algebra associated with an exact factorization $G=HK$ is biperfect if and only if both $H$ and $K$ are perfect and self-normalizing.

\subsection{Proof of Corollary~\ref{CorBiperfect}}
\ \vspace{1mm}

The following is an immediate consequence of Theorem~\ref{ThmExact}.

\begin{corollary}\label{ThmBiperfect}
Let $G=HK$ be an exact factorization of a finite simple group $G$ with proper subgroups $H$ and $K$. If both $H$ and $K$ are perfect, then, interchanging $H$ and $K$ if necessary, one of the following holds:
\begin{enumerate}[{\rm (a)}]
\item $G=\A_n$ with $n\geqslant60$, $H$ is a perfect group of order $n$, and $K=\A_{n-1}$;
\item $G=\A_{q+1}$ with $q=2^f\geqslant8$, $H=\PSL_2(q)$, and $K=\A_{q-2}$;
\item $(G,H,K)$ lies in row~\emph{6,~7,~9,~19} or~\emph{22} of Table~$\ref{TabXia2}$ with $\calO=1$.
\end{enumerate}
\end{corollary}

Now we can prove Corollary~\ref{CorBiperfect}, which is an answer to a question arising from the work of~\cite{EGGS2000} on constructing biperfect semisimple Hopf algebras.

\begin{proof}[Proof of Corollary~$\ref{CorBiperfect}$]
Let $G$ be a finite simple group. For an exact factorization $G=HK$ with both $H$ and $K$ perfect and self-normalizing, the triple $(G,H,K)$ lies in parts~(a)--(c) of Corollary~\ref{ThmBiperfect} since $H$ and $K$ are perfect. To complete the proof, we only need to show that none of parts~(a)--(c) of Corollary~\ref{ThmBiperfect} is possible for $H$ and $K$ to be both self-normalizing except for row~22 of Table~\ref{TabXia2} with $\calO=1$.

First suppose that part~(a) occurs. Then $H$ is a regular subgroup of $G=\A_n$, and so $\Nor_G(H)=\Hol(H)\cap G$, where $\Hol(H)=H{:}\Aut(H)$ is the holomorph of $H$. Since $\Nor_G(H)=H$, we have
\[
\frac{|H|}{|\Z(H)|}\leqslant|\Aut(H)|=\frac{|H{:}\Aut(H)|}{|H|}=\frac{|\Hol(H)|}{|\Nor_G(H)|}=\frac{|\Hol(H)|}{|\Hol(H)\cap G|}=\frac{|\Hol(H)G|}{|G|}\leqslant\frac{|\Sy_n|}{|\A_n|}=2.
\]
This implies that $\Z(H)\geqslant H'$, which contradicts the condition $H'=H$.

Next suppose that either part~(b) occurs or $(G,H,K)$ lies in row~6,~7 or~9 of Table~\ref{TabXia2} with $\calO=1$. Then $G=\A_n$ and $K=\A_{n-k}$ with $2\leqslant k\leqslant5$. This implies that $\Nor_G(K)=(\A_k\times\A_{n-k}).2$, contradicting the condition $\Nor_G(K)=K=\A_{n-k}$.

Finally, if $(G,H,K)$ lies in row~19 of Table~\ref{TabXia2} with $\calO=1$, then we conclude from~\cite{CCNPW1985} that $\Nor_G(H)=2^6{:}(3\times\A_5).2>H$, again a contradiction.
\end{proof}

\section{Transitive simple subgroups of quasiprimitive permutation groups}\label{SecXia8}

In this section, we characterize finite quasiprimitive permutation groups containing a transitive simple subgroup.
The O'Nan-Scott-Praeger Theorem \cite{Praeger92} divides quasiprimitive permutation groups into eight types.
For our purpose here, we categorize the eight types of quasiprimitive groups $X$  into four classes:
\begin{enumerate}[{\rm(i)}]
\item $\Soc(X)$ is abelian, and $X$ is affine;

\item $X$ has a subnormal regular subgroup, and $X$ is of type HS, HC, SD, CD or TW;

\item $X$ is of product action type;

\item $X$ is an almost simple group.
\end{enumerate}

Our analysis proceeds by a series of lemmas.
The first one is for affine type.

\begin{lemma}\label{HA>simple}
Let $X$ be a quasiprimitive group of affine type, and let $L$ be a transitive simple subgroup of $X$.
Then $X=\AGL_3(2)$, and $L=\PSL_2(7)$.
\end{lemma}
\begin{proof}
Assume that $M=\Soc(X)=\ZZ_p^m$ with $p$ prime.
Then $M$ is regular on $\Omega$, and so $|M|=p^m$ divides $|L|$.
Thus $L$ is a simple group and has a subgroup of index $p^m$, which is the stabilizer of a point.
By \cite[Theorem~1]{Guralnick1983}, $L$ is one of the following groups:
\begin{align*}
&\mbox{$\A_{p^m}$,\,\ $\PSL_d(q)$ with $(q^d-1)/(q-1)=p^m$,\,\ $\PSL_2(11)$ with $p^m=11$,}\\
&\mbox{$\PSU_3(3)$ with $p^m=3^3$,\,\ $\M_{11}$
with $p^m=11$\,\ or $\M_{23}$ with $p^m=23$.}
\end{align*}
In this case, $M$ is the unique minimal normal subgroup of $X$, and hence $L$ acts faithfully on $M=\ZZ_p^m$.
Thus $L\not=\PSL_2(11)$, $\PSU_3(3)$, $\M_{11}$ or $\M_{23}$.
If $L=\A_{p^m}$, then $L\leqslant\GL_m(p)$, which is not possible since $|L|=p^m!/2>|\GL_m(p)|$.
Thus $L=\PSL_d(q)$ with ${q^d-1\over q-1}=p^m$, and so $L$ is 2-transitive on $V$.
That is to say, $L$ is a 2-transitive subgroup of the affine group $\AGL_m(p)$.
By \cite{LPS1987}, we have that $L=\PSL_2(7)$ and $X=\AGL_3(2)$.
\end{proof}

The next lemma is to deal with the second class of groups.

\begin{lemma}\label{SD>simple}
Suppose that $X$ has a subnormal regular subgroup, namely, $X$ is of type {\rm HS, HC, SD, CD} or {\rm TW}, and suppose that $X$ contains a transitive simple group $L$.
Then $X$ is of type $\HS$ or {\rm SD}, $\Soc(X)=L\times L$, and $L$ is regular on $\Omega$.
\end{lemma}
\begin{proof}
Let $M=\Soc(X)$.
Then $M$ is transitive on $\Omega$ since $X$ is quasiprimitive on $\Omega$, and
\[M=T_1\times\dots\times T_m=T^m,\]
where $m\geqslant2$, and $T_i\cong T$ is simple.
Let $N=T^\ell\lhd M$ with $\ell\geqslant1$ be regular on $\Omega$.

Suppose that $\ell\geqslant 2$.
Then $|N|=|T|^\ell$ divides $|L|$ as $L$ is transitive on $\Omega$.
In this case, $\ell\geqslant m/2$ by Praeger' theorem~\cite[Theorem~1]{Praeger92}.
%Without loss of generality, assume that $N=T_1\times\dots\times T_\ell$.
Now $L$ acts faithfully on the set of the direct factors $\{T_1,\dots,T_m\}$, and thus $L\leqslant\Sy_m$.
Let $r$ be an odd prime divisor of $|T|$.
Then $r^\ell$ divides $|N|=|T|^\ell$, and $|L|_r\leqslant|\Sy_m|_r< r^{m\over r-1}$.
Since $|N|$ divides $|L|$, we have that $r^{m\over2}\leqslant r^\ell< r^{m\over r-1}\leqslant r^{m\over2}$, which is a contradiction.

We thus have $\ell=1$, and so $X$ is of type HS or type SD with $M=\Soc(X)=T\times T$.
Moreover, $T\times T\lhd X\leqslant (T\times T).(\Out(T)\times\Sy_2)$.
Since $L$ is a nonabelian simple subgroup of $X$, we have $L< T\times T$.
As $L$ is transitive and $T$ is regular on $\Omega$, it follows that $L\cong T$.
\end{proof}

The third class of groups are dealt with in the following lemma.

\begin{lemma}\label{PA>simple}
Let $X$ be a quasiprimitive permutation group on $\Omega$ of product action type with socle $M$, and let $L\leqslant X$ be transitive and simple. Then one of the following holds:
\begin{enumerate}[{\rm (a)}]
\item
$X$ is primitive of degree $d^2$ such that $M=T\times T$ and $M_\omega=S\times S$, where $(L,T,S,d)$ lies in the following table:
\[
\begin{array}{|llll|} \hline
L & T & S & d \\ \hline
\A_6 & \A_6 & \A_5 & 6\\ \hline
\M_{12} & \M_{12} & \M_{11} &12\\
 & \A_{12} & \A_{11} &\\ \hline
\Sp_4(q) & \Sp_4(q) & \Sp_2(q^2).2 & q^2(q^2-1)/2\\
 \textup{($q$ even)} &\A_d & \A_{d-1} & \\
 &\Sp_{4r}(q^{1/r})& {\rm O}_{4r}^-(q^{1/r})&  \\ \hline
\POm_8^+(q) & \POm_8^+(q) &\Omega_7(q) & q^3(q^4-1)/(2,q-1)\\
 & \A_d &\A_{d-1} &  \\ \hline
\Omega_8^+(2) & \Sp_8(2) & {\rm O}_8^-(2) & 120 \\
\hline
\end{array}
\]
\item $X$ is imprimitive, $M_\omega=(S'\times S').2$, and $(L,T,S,L_\omega,M_\omega)$ lies in the following table:
\[
\begin{array}{|lllll|} \hline
L & T & S & L_\omega& M_\omega \\ \hline
\Sp_4(q) & \Sp_4(q) & \Sp_2(q^2).2 & \D_{2(q^2+1)}&\Sp_2(q^2)^2.2\\
 \textup{($q$ even)} &\Sp_{4r}(q^{1/r})&
 {\rm O}_{4r}^-(q^{1/r}) &\D_{2(q^2+1)} &\Omega_{4r}^-(q^{1/r})^2.2\\ \hline
\Omega_8^+(2) & \Sp_8(2) & {\rm O}_8^-(2) & \PSU_3(3)&\Omega_8^-(2)^2.2\\\hline
\end{array}
\]
\end{enumerate}
\end{lemma}

\begin{proof}
If $X$ is primitive on $\Omega$, then by \cite[Lemma\,4.2]{Baddeley-Praeger-2003}, the lemma is true.

Suppose that $X$ is not primitive on $\Omega$. Let $\calB$ be a maximal imprimitive block system for $X$ acting on $\Omega$. Then $X$ acts faithfully on $\calB$ as $X$ is quasiprimitive on $\Omega$. It follows that $X^\calB$ is a primitive group with a transitive simple subgroup $L^\calB\cong L$.
By \cite[Theorem~1]{Praeger2003} and Lemma~\ref{SD>simple}, either $X^\calB$ has type PA or $X^\calB$ is of type SD with $\Soc(X)=L\times L$.
However, the latter implies that $|\calB|=|L|\geqslant|\Omega|$, contradicting that $\calB$ is an imprimitive block system on $\Omega$.
Thus $X^\calB$ is a primitive group of type PA.
Pick a block $\Delta\in\calB$, and let $d=|\Delta|$.
By \cite[Lemma\,4.2]{Baddeley-Praeger-2003}, we have that $L<M=T_1\times T_2$ with $T_1\cong T_2\cong T$ simple, and $M_\Delta=S_1\times S_2$ with $S_1\cong S_2\cong S$ simple, and $(L,T,S,d)$, as in the table of part~(a).

Let $\omega\in\Delta$.
Then $M_\Delta=L_\Delta M_\omega$.
As $L$ is transitive on $\Omega$, we have $M=LM_\omega$, and so
\[|M_\omega|={|M|\over|L|}|L_\omega|={|M_\Delta|\over|L_\Delta|}|L_\omega|.\]

Suppose that $T_1$ is semiregular on $\Omega$.
Since $L$ is transitive, $|T_1|$ divides $|L|$.
As $L\lesssim T$, we have that $T_1\cong L$ is regular on $\Omega$, which is a contradiction since $X$ is of product action type.
Thus $T_1$ is not semiregular on $\Omega$, and $M_\omega\cap T_1\not=1$.
Similarly, $M_\omega\cap T_2\neq1$.
As $X$ is quasiprimitive, we have $MX_\omega=X$, and $X_\omega$ is transitive on $\{T_1,T_2\}$, so that
\[M_\omega\cap T_1\cong M_\omega\cap T_2.\]
Then $M_\omega\leqslant M_\Delta=S_1\times S_2$, and it follows that $M_\omega\cap T_i=M_\omega\cap S_i$ where $i\in\{1,2\}$. Consequently, $M_\omega\cap S_1\cong M_\omega\cap S_2$.

Noticing that $L\lesssim T$ and $|M_\omega|={|M|\over|L|}|L_\omega|$, we have
\[|M_\omega|={|M|\over|L|}|L_\omega|={|T_1\times T_2|\over|L|}|L_\omega|=|S_1|{|T_1|\over|S_1|}{T_2|\over|L|}|L_\omega|.\]
Thus $|M_\omega|$ is divisible by all prime divisors of $|S|$.
%
% \[{|M_\Delta|\over|M_\o|}={|S_1\times S_2|\over|M_\o|}=|S_1|.{|T_1||T_2|\over|S_1||L|}.\]
%Thus $M_\o\cap S_1$ is a subgroup of $S_1$ of order divisible by ${|T_1||T_2|\over|S_1||L|}$.
As $(M_\omega\cap S_1)\times(M_\omega\cap S_2)\unlhd M_\omega$, we have
\[1\not=M_\omega\cap S_2\unlhd M_\omega/(M_\omega\cap S_1).\]
Let $r$ be a prime divisor of $|S|$.
Then $r$ divides $|M_\omega|$, and so $r$ divides $|M_\omega\cap S_1|=|M_\omega\cap S_2|$ or $|M_\omega/(M_\omega\cap S_1)|$.
It follows that $r$ divides $|M_\omega/(M_\omega\cap S_1)|$, that is, $M_\omega/(M_\omega\cap S_1)$ is isomorphic to a subgroup of $S$ of order divisible by all prime divisors of $|S|$.
Assume that $S$ is a simple group.
By \cite[Theorem~1]{Praeger92},
$M_\omega$ is a
subdirect product of $M_\Delta$, and hence $M_\omega\cap S_i=M_\omega/(M_\omega\cap S_j)$ with $\{i,j\}=\{1,2\}$.
Consequently, $M_\omega=(M_\omega\cap S_1)\times(M_\omega\cap S_2)$,
which leads to a contradiction that $M_\omega=M_\Delta$.
Thus $S$ is not simple, namely,
$S=\Sp_2(q^2).2$, $\GO^-_{4r}(q^{1/r})$ or $\GO^-_8(2)$.
Since $M_\omega$ is a
subdirect product of $M_\Delta$, we deduce that
\[\mbox{$S'=M_\omega\cap S_i\lhd M_\omega/(M_\omega\cap S_j)=S$,}\]
where $\{i,j\}=\{1,2\}$. This implies that
$M_\omega=(S'\times S').2$.
First, let $S'=\Sp_2(q^2)$ or $\Omega^-_{4r}(q^{1/r})$.
Since $|L_\omega|=|M_\omega||L|/|T|^2$,
we calculate that $|L_\omega|=2(q^2+1)$.
Consulting \cite[Table~8.14]{BHR2013},
we deduce that $L_\omega=\D_{2(q^2+1)}$.
Now, let $S=\GO^-_8(2)$.
It follows that $|L_\omega|=|M_\omega||L|/|T|^2=6048$.
Then computation in \magma~\cite{BCP1997} shows that
$L_\omega=\PSU_3(3)$.
This completes the proof.
%Notice that $S\cong\A_5$, $\M_{11}$, $\PSL_2(q^2).2$ with $q$ even, or $\Omega_7(q)$.
%Since $1\not=M_\o\cap S_2\lhd M_\o/(M_\o\cap S_1)$, it follows that $M_\o/(M_\o\cap S_1)$ is not almost simple.
%Thus $M_\o/(M_\o\cap S_1)$ is a non-simple proper subgroup of $S$.
%By \cite[Corollary\,5]{LPS2000}, the only possibility is $S=\A_{d-1}$, and $\A_k\lhd M_\o/(M_\o\cap S_1)\leqslant\Sy_k\times\Sy_{d-1-k}$ such that there is no prime between $k$ and $d-1$.
%%{\color{red}
%This is not possible since $M_\o$ should have a normal subgroup $(M_\o\cap S_1)\times(M_\o\cap S_2)$ with $M_\o\cap S_1\cong M_\o\cap S_2$.
%Therefore, we conclude that $X$ is primitive on $\Omega$.
\end{proof}

The above three lemmas lead to a description of finite quasiprimitive permutation groups with a transitive simple subgroup as follows.

\begin{theorem}\label{qp=as}
Let $X$ be a quasiprimitive permutation group on $\Omega$, and let $L$ be a transitive simple subgroup of $X$.
Then one of the following holds:
\begin{enumerate}[{\rm (a)}]
\item $X=\AGL_3(2)$, and $L=\PSL_2(7)$;

\item $X$ is of type $\HS$ or {\rm SD}, $\Soc(X)=L\times L$, and $L$ is simple and regular on $\Omega$;

\item $X$ is of product action type, and all possibilities are listed in Lemma~$\ref{PA>simple}$;

\item $X$ is almost simple, and either $L\lhd X\leqslant\Aut(L)$, or $X=LX_\a$ is a factorization.
\end{enumerate}
\end{theorem}
\begin{proof}
Let $M$ be the socle of $X$.
If $M$ is not a nonabelian simple group, then one of parts~(a)--(c) of the theorem is satisfied by Lemmas~\ref{HA>simple}--\ref{PA>simple}, respectively.

Assume that $M$ is nonabelian simple.
Then $X$ is almost simple.
If $L=\Soc(X)$, then $L\lhd X\leqslant \Aut(L)$.
If $L$ is not normal in $X$, then $X=LX_\a$ is a factorization.
\end{proof}

Now we are ready to prove Theorem~\ref{qp>simple-gps}, which presents a classification of quasiprimitive groups containing a regular simple group.

\vskip0.1in
{\bf Proof of Theorem~\ref{qp>simple-gps}:}

Let $X$ be a quasiprimitive permutation group on $\Omega$, and let $L\leqslant X$ be a nonabelian simple group which is regular on $\Omega$.
Let $M=\Soc(X)$, and $\omega\in \Omega$.
Then $X$ and $L$ satisfy Theorem~\ref{qp=as}.
Moreover, as $L$ is regular on $\Omega$, $X$ is neither affine nor of product action type.

If $X$ and $L$ satisfy part~(b) of Theorem~\ref{qp=as}, then $X$ is of type HS or SD.
Thus the stabilizer $M_\omega\cong L$ acts on $L$ by conjugation, and $L$ is a regular subgroup, as in Theorem~\ref{qp>simple-gps}~(b).

Assume that $X$ is an almost simple group.
If $L\lhd X$, then $X\leqslant\Aut(L)$, as in part~(a).
Assume further that $L$ is not normal in $X$.
Then $X=LX_\a$ is a core-free exact factorization with one factor $L$ being nonabelian simple.
Such a factorization $X=LX_\a$ is classified in Theorem~\ref{ThmExact} for $X_\a$ nonsolvable and \cite[Theorem\,3]{BL2021} for $X_\a$ solvable, which leads to part~(c) of Theorem~\ref{qp>simple-gps}.
\qed

\section*{Acknowledgments}

The authors would like to thank the anonymous referee for careful reading and valuable suggestions to this paper. The first author acknowledges the support of NNSFC grants no.~11771200 and no.~11931005. The second author acknowledges the support of NNSFC grant no.~12061083, NSFC grant of Yunnan Province no.~202101AT070023, and Scientific research and innovation fund of Yunnan University no.~ST20210105.

\end{document}